\documentclass{tac}
\usepackage{xy}
\usepackage{amssymb}
\usepackage{plain}
\usepackage{amsfonts}
\usepackage[centertags]{amsmath}
\xyoption{all}
\def\mc#1{\mathcal {#1}}
\def\A{\mc A}
\def\B{\mc B}
\def\C{\mc C}

\def\E{\mc E}

\def\I{\mc I}
\def\J{\mc J}
\def\M{\mc M}

\def\X{\mc X}
\def\hs#1{\hskip#1mm}

 \author{Walter Tholen and Leila Yeganeh}
\address{Department of Mathematics and Statistics\\ York University, Toronto, Canada}

 \title{The comprehensive factorization of Burroni's $\mathbb T$-functors}


 \eaddress{ tholen@yorku.ca, yeganehleila39@gmail.com}

\begin{document}

 \maketitle

 %{\em Dedicated to Robert Rosebrugh, Founder of ``Theory and Applications of Categories''}

 \bigskip

 \begin{abstract} erExpanding on the comprehensive factorization of functors internal to a category $\C$, under fairly mild conditions on a monad $\mathbb T$ on $\C$ we establish that this orthogonal factorization system exists even in Burroni's category ${\sf Cat}(\mathbb T)$ of (internal) $\mathbb T$-categories and their functors. This context provides for some expected applications and some unexpected connections. For example, it lets us deduce that the comprehensive factorization is also available for functors of Lambek's multicategories. In topology, it leads to the insight that the role of discrete cofibrations is played by perfect maps, with the comprehensive factorization of a continuous map of Tychonoff spaces given by its fibrewise compactification.
\end{abstract}

 \section{Introduction}
There are two important factorization systems in seemingly disjoint environments that may serve as a motivation for the general results of this paper.

 First, in category theory, as recorded in \cite{Gray1969}, it is due to Bill Lawvere to interpret the set-theoretic comprehension schema (whereby, roughly, given a property, there is a set consisting exactly of the elements having that property) as the adjunction
\begin{center}
$\xymatrix{{\bf Set}/B\ar@/^8pt/[rr] &\perp & 2^B\ar@/^6pt/[ll]
}$	
\end{center}
whose left adjoint associates with a $B$-valued map (the characteristic function of) its image. Lawvere's general categorical account of the comprehension schema appeared in \cite{Lawvere1970}. But already in his article, replacing the set $B$ above by a (small) category $\B$, John Gray established a categorification of the above adjunction,
\begin{center}
$\xymatrix{{\bf Cat}/\B \ar@/^8pt/[rr] &\perp & {\bf Cat}^{\B}\;,\ar@/^6pt/[ll]}$	
\end{center}
as follows. Its right adjoint embeds ${\bf Cat}^{\B}$ fully into ${\bf Cat}/\B$ by the dual Grothendieck construction, considering ${\bf Cat}$-valued functors on $\B$ equivalently as cofibred categories over $\B$. The left adjoint assigns to an object $F:\A\longrightarrow\B$
of ${\bf Cat}/\B$ the functor $\B\longrightarrow{\bf Cat}$ which maps a $\B$-object $B$ to the comma category $F/B$. It may equivalently be described as the left Kan extension along $F$ of the functor $\A\longrightarrow{\bf Cat}$ with constant value the terminal category $\sf 1$. The left adjoint $\pi_0$ of the ``discrete embedding" $\bf Set\hookrightarrow\bf Cat$ allows one to ``scale back" the codomain of the Gray adjunction to $\bf Set^{\B}$, in the form of the composite adjunction
\begin{center}
$\xymatrix{{\bf Cat}/\B \ar@/^8pt/[rr] &\perp & {\bf Cat}^{\B}\ar@/^6pt/[ll]\ar@/^8pt/[rr]^{\pi_0(-)} & \perp&{\bf Set}^{\B}\;,\ar@/^6pt/[ll]
}$
\end{center}
for which Gray credits Lawvere.
Its right adjoint presents the $\bf Set$-valued functors on $\B$ equivalently as {\em discrete (Grothendieck) cofibrations} over $\B$ (nowadays more often called discrete opfibrations over $\B$). Since the left-adjoint functor $\pi_0$ preserves the left Kan extension above, the left adjoint of the composite adjunction assigns to a $\B$-valued functor $F$ the left Kan extension along $F$ of the functor $\A\longrightarrow\bf Set$ with constant value the singleton set $1$. As a pointwise left Kan extension, it may be computed as
\begin{center}
$\xymatrix{\B\ar[rr]^{{\rm Yoneda}\quad} && {\bf Set}^{\B^{\rm op}}\ar[rr]^{F^{\rm op}(-)} && {\bf Set}^{\A^{\rm op}}\ar[rr]^{\rm colim} && {\bf Set}\;,
}$
\hfil
$B\mapsto {\rm colim}\,\B(F-,B)\,.$	
\end{center}
Considering its ``category of elements" $\E$ with its discrete cofibration $P:\E \longrightarrow \B$, one sees that the composite adjunction lets $F$ factor universally through $P$. The celebrated paper \cite{StreetWalters1973} shows that one obtains an orthogonal factorization system in $\bf Cat$.

 There are various generalized categorical contexts in which the Street-Walters comprehensive factorization has been established. We refer the reader particularly to the fairly recent paper \cite{BergerKaufmann2017}. The most important generalization with respect to this work is the fact that it may be established for functors internal to a category with pullbacks and reflexive coequalizers, these having to be stable under pullback: see \cite{Johnstone2002}.

 Here is our second motivating example of an important factorization system, this one appearing in general topology. It is well known (see \cite{Whyburn1966}) that a continuous map $f:X\longrightarrow Y$ of Tychonoff (= completely regular Hausdorff) spaces may be universally factored through a {\em perfect} map $p:E\longrightarrow Y$. For arbitrary spaces, perfect maps are best defined as stably closed (= proper \cite{Bourbaki1989}) and Hausdorff-separated (\cite{James1989}) continuous maps. In the Tychonoff context they are characterized as the maps for which their Stone-$\rm{\check{C}}$ech naturality squares are pullback diagrams (\cite{HenriksenIsbell1958}, \cite{Herrlich1972}). This characterization suggests to factor $f$ in the general style of \cite{Ringel1970} and \cite{CassidyHebertKelly1985}, by taking $p$ to be the pullback of $\beta f$ along the the Stone-$\rm{\check{C}}$ech compactification of $Y$, as shown in the diagram
\begin{center}
$\xymatrix{X\ar@/^12pt/[drrr]^f\ar@/_15pt/[ddr]_{\beta_X}\ar@{-->}[dr]^r & & &\\
& E=\beta X\times_{\beta Y}Y\ar[d]\ar[rr]^{\quad p} && Y\ar[d]^{\beta_Y}\\
& \beta X\ar[rr]^{\beta f} && \beta Y
}$	
\end{center}
The paper \cite{Tholen1999} provides a very general categorical context in which the universal factorization through a perfect morphism may be established in this fashion.

 In this paper we provide a common generalization for the notions of discrete cofibrations of categories and perfect maps of topological spaces, under the roof of Burroni's $\mathbb T$-categories and their functors (\cite{Burroni1971}; see also \cite{Hermida2000}, \cite{Zawadowski2011}). For $\mathbb T$ the identity monad of a category $\C$ with pullbacks, they generalize categories and functors internal to $\C$, but for an arbitrary monad $\mathbb T$ they reach far beyond that context. In particular, as observed already by Burroni, for $\mathbb T$ the ultrafilter monad on $\bf Set$, {\em ordered} $\mathbb T$-categories describe topological spaces in terms of ultrafilter convergence, as presented in \cite{Barr1970}, following Manes' equational presentation of compact Hausdorff spaces \cite{Manes1969}. Our principal result establishes a generalization of the comprehensive factorization for internal functors ($\mathbb T={\rm Id}_{\C}$) to $\mathbb T$-functors, under fairly mild conditions on a general monad $\mathbb T$ on $\C$. For a category $\C$ with pullbacks and coequalizers of reflexive pairs, these being stable under pullback, we require only that the endofunctor $T$ of $\C$ preserve the reflexive coequalizers. To our surprise, no preservation of pullbacks by $T$, or any further ``cartesianness" of $\mathbb T$, is needed for the establishment of the {\em Comprehensive Factorization Theorem (CFT)}, as in Theorem \ref{compfact}.

 Our proof of this theorem is quite intricate and long, as we skip only very few routine details in our argumentation. In fact, we start off with a very elaborate and detailed presentation of $\mathbb T$-categories, not assuming any familiarity by the reader with Burroni's work (Section 2), followed by the mentioning of the principal examples in Section 3 which, as another pillar, include Lambek's multicategories (\cite{Lambek1969}; see also \cite{Leinster2004}. In Theorem \ref{topologicity} we expand on Burroni's original presentation of the category ${\sf Cat}(\mathbb T)$ of $\mathbb T$-categories as a fibred category over $\C$ and, generalizing a result of \cite{Hoffmann1972} for $\C={\bf Set}$, show that the ``object-of-objects" functor to $\C$ is actually {\em small-topological} when $\C$ is complete. This theorem gives, in particular, a general guide on how to compute limits in ${\sf Cat}(\mathbb T)$.

 For the CFT, however, in the absence of the preservation of pullbacks in $\C$ by $T$, we need a presentation of pullbacks in ${\sf Cat}(\C)$ that is more explicit than the one provided by Theorem \ref{topologicity}, in order to be able to deduce easily from it the pullback stability of the class of perfect $\mathbb T$-functors (= discrete cofibrations). This presentation is provided in Proposition \ref{pbstability} and its long proof. Generally, the pullback stability is one of the three needed ingredients for a class $\M$ of morphisms in any category $\A$ to possess an orthogonal
factorization partner $\E$. The other two are the closure of $\M$ under composition and the existence of an adjunction
\begin{center}
$\xymatrix{\A/B\ar@/^8pt/[rr] & \perp &\M/B\ar@/^6pt/[ll]
}$	
\end{center}
for every object $B$, where the left adjoint is the obvious full embedding into the comma category $\A/B$; see Proposition \ref{factsystems}. For $\A={\sf Cat}(\mathbb T)$ and $\M=\{$perfect $\mathbb T$-functors$\}$, this leaves us with having to establish this adjunction, just as in the two guiding examples mentioned above, in order to complete the proof of the CFT. This is done in Theorem \ref{reflective}, which carries the main burden toward the proof of the CFT.

 Our presentation is strictly one-dimensional. For stepping into the next dimension and present ${\sf Cat}(\mathbb T)$ as a 2-category, one needs the preservation of pullbacks in $\C$ by the endofunctor $T$. The proof is included in the forthcoming paper \cite{TholenYeganeh2021}.

 %We have added this section as an appendix to this paper, for the sake of greater completeness of our presentation of $\mathbb T$-categories.
\medskip
{\em Acknowledgement} We thank Dirk Hofmann for providing the example recorded in Examples \ref{examplefacts}(3).

\section{Preliminaries}
{\em Throughout the paper, we work in a category $\C$ with pullbacks and let $\mathbb T$ be a monad on $\C$}.

 \medskip
Recall that a {\em monad} $\mathbb{T} = (T, \eta, \mu)$ on $\C$ is given by a functor $\xymatrix{T:\C\ar[r]&\C}$ and natural transformations $\xymatrix{\mu:TT\ar[r]&T}$ and $\xymatrix{\eta:{\rm Id}_\C\ar[r]&T}$, called the multiplication and the unit of $\mathbb T$ respectively, satisfying
	$$
	\mu\cdot T\eta=T=\mu\cdot \eta T \qquad\text{and} \qquad \mu\cdot T\mu=\mu \cdot \mu T.
	$$
%equivalently, the following diagrams commute.
\begin{center}
	$\xymatrix{T\ar[rrd]_{1_T}\ar[rr]^{\eta T}&&TT\ar[d]^{\mu}&&T\ar[ll]_{T\eta}\ar[lld]^{1_T}\\ &&T&&}\qquad \qquad \xymatrix{TTT\ar[d]_{T\mu}\ar[rr]^{\mu T}&&TT\ar[d]^{\mu}\\ T\ar[rr]_{\mu}&&T}$
\end{center}
% We recall that for a given monad $\mathbb{T} = (T, \eta, \mu)$ on $\C$,
A $\mathbb{T}$-{\em algebra} $(A, a)$ is an object $A$ of $\C$ equipped with a morphism
$\xymatrix{a:TA\ar[r]&A}$ satisfying
$$
1_{A}=a\cdot \eta_{A}\qquad \text{and} \qquad a\cdot Ta=a\cdot \mu_{A}.
$$
A $\mathbb{T}$-{\em homomorphism} $\xymatrix{h:(A, a)\ar[r]&(B, b)}$ is a $\C$-morphism $\xymatrix{f:A\ar[r]&B}$ such that
\begin{center}
$f\cdot a=b\cdot Tf$
\end{center}
The category ${\sf Alg}(\mathbb T)$ of $\mathbb{T}$-algebras and $\mathbb{T}$-homomorphisms, often denoted by $\C^{\mathbb{T}}$ is known as the {\em Eilenberg-Moore category of} $\mathbb{T}$ (\cite{Mac Lane}).

 %Let $\C$ be a category with pullbacks and $\mathbb{T} = (T, \eta, \mu)$ be a monad on $\C$.
\medskip

 Fixing our notation, we first recall Burroni's fundamental definitions \cite{Burroni1971}.
\begin{itemize}
\item[1.] A $\mathbb T$-{\em graph} $A=(A_{0}, A_{1}, d_{0}^A, c_{0}^A)$ is simply a span
\begin{center}
$\xymatrix{ & A_1\ar[ld]_{d_0^A}\ar[rd]^{c_0^A} &\\
TA_0 & & A_0
}$	
\end{center}
of morphisms in $\C$. One calls $A_0$ the {\em object of objects} (or of {\em vertices}) of $A$ and $A_1$ the {\em object of morphisms} (or of {\em edges}) of $A$; $d_0=d_0^A$ is its {\em domain} (or {\em source}) {\em morphism} and $c_0=c_0^A$ its {\em codomain} (or {\em target}) {\em morphism}.

 A {\em morphism
%$Gr(\mathbb{T})$ is a category whose objects are quaternaries $\A=(A_{0}, A_{1}, d_{0}, c_{0})$(called $\mathbb{T}$-graph) with $\C$-morphisms $d_{0}, c_{0}$, called domain and codomain, respectively; and morphisms
$\xymatrix{f=(f_0, f_1):A\ar[r]&B}$
of $\mathbb T$-graphs} has an {\em object part} $f_0:A_0\longrightarrow B_0$ and a {\em morphism part} $f_1:A_1\longrightarrow B_1$ such that the diagram
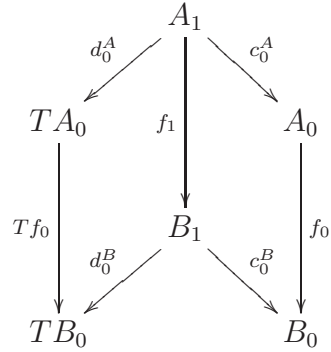
\begin{center}
	$\xymatrix{&A_{1}\ar[ld]_{d_{0}^A}\ar[dd]_{f_1}\ar[rd]^{c_{0}^A}&\\
	TA_{0}\ar[dd]_{Tf_{0}}&&A_{0}\ar[dd]^{f_{0}}\\
	& B_1\ar[ld]_{d_{0}^B}\ar[rd]^{c_{0}^B}	 & \\
	TB_{0}&&B_{0}\\
	 }$
\end{center}
commutes in $\C$. Composed as in $\C$, the morphisms give us the category
$${\sf Gph}(\mathbb T)$$ of $\mathbb T$-graphs in $\C$, which comes with the (not necessarily faithful) {\em object-of-objects functor}
$(-)_0: {\sf Gph}(\mathbb T)\longrightarrow \C$.%, which comes with a (not necessarily faithful) object functor
%$$(-)_0: {\sf Gra}(\mathbb T)\longrightarrow \C.$$
%,$A_0, A_1$ are called object of objects and objects of morphisms, respectively; and $f_0, f_1$ are called object part and morphism part of the morphism $f$, respectively.
\item[2.] Trivially, every object $X$ in $\C$ gives rise to the {\em discrete} $\mathbb T$-graph
$DX=(X,X,\eta_{X},1_{X}),$
and every $\C$-morphism $h:X\longrightarrow Y$ gives a morphism $Dh=(h,h):DX\longrightarrow DY$ of $\mathbb T$-graphs. % one has a functor $D:\C\longrightarrow {\sf Gra}({\mathbb T})$.
For a $\mathbb T$-graph $A$, we write $$D_0A:=D(A_0)=(A_0,A_0,\eta_{A_0},1_{A_0}),$$ and likewise for morphisms.

 A {\em pointed $\mathbb T$-graph} $(A,i^A)$ is a $\mathbb T$-graph $A$ equipped
%Pointed $\mathbb{T}$-graphs, $Gr_{\bullet}(\mathbb{T})$, whose objects are $\mathbb{T}$-categories $\A=(A_{0}, A_{1}, d_{0}, c_{0})$ with
with a morphism $$\xymatrix{(1_{A_0},i^A):D_0A\ar[r]&A}$$ of $\mathbb T$-graphs. Hence, the {\em insertion of identities} $i=i^A$ is simply a $\C$-morphism $i:A_0\longrightarrow A_1$ satisfying
\begin{equation}
d_0\cdot i=\eta_{A_0}\qquad\text{and}\qquad c_0\cdot i=1_{A_0}.
\end{equation}
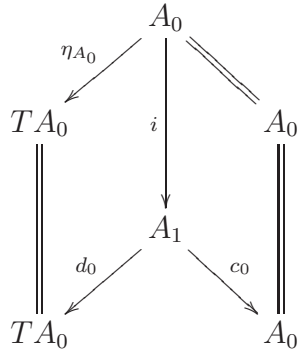
\begin{center}
	$\xymatrix{&A_{0}\ar[ld]_{\eta_{A_0}}\ar[dd]_{i}\ar@{=}[rd]&\\
	TA_{0}\ar@{=}[dd]&&A_{0}\ar@{=}[dd]\\
	& A_1\ar[ld]_{d_{0}}\ar[rd]^{c_{0}}	 & \\
	TA_{0}&&A_{0}\\
	 }$
\end{center}
%\begin{center}
%	$\xymatrix{&A_0 \ar[d]_{i} \ar@/_1.5pc/[ldd]_{\eta_{A_{0}}}\ar@/^1.5pc/[rdd]^{1} & \\ & A_{1}
%		 \ar[ld]_{d_{0}}\ar[rd]^{c_{0}} & \\
%		TA_{0}& & A_{0}}$
%\end{center}
A {\em morphism $f:(A,i^A)\longrightarrow(B,i^B)$ of pointed ${\mathbb T}$-graphs} is a morphism $f:A\longrightarrow B$ of $\mathbb T$-graphs such that
%and $Gr_{\bullet}(\mathbb{T})(\A, \B)=Gr(\mathbb{T})(\A, \B)$
% with
\begin{center}
	$\xymatrix{D_0A\ar[r]^{D_0f}%\ar@{}[rd]|{\circlearrowright}
	\ar[d]_{(1_{A_0},i^A)}&D_0B\ar[d]^{(1_{B_0},i^B)}\\ A\ar[r]_{f}&B}$
\end{center}
commutes in ${\sf Gph}({\mathbb T})$; this simply means $f_1\cdot i^A=i^B\cdot f_0$ in $\C$. The pointed graphs and their morphisms form the category
$${\sf Gph}_{\bullet}(\mathbb T).$$ %which comes with a (not necessarily faithful) object functor
%$(-)_0: {\sf Gra}_{\bullet}(\mathbb T)\longrightarrow \C.$

 Note that the discrete $\mathbb T$-graph $DX$ is pointed, by $1_X$, and that one has an adjunction
$$D\dashv(-)_0:{\sf Gph}_{\bullet}(\mathbb T)\longrightarrow\C;$$
in fact $D$ is left-adjoint right-inverse to $(-)_0$.
\item[3.]
To be able to consider a categorical composition on a pointed $\mathbb T$-graph $(A,i)$, one forms the object $A_2=TA_1\times_{TA_0}A_1$ of {\em composable pairs} and then the object $A_3=TA_2\times_{TA_1}A_2$ of {\em composable triples} of $A$, via the pullback diagrams
\begin{center}
$\xymatrix{TA_1\ar[d]_{Tc_0} & A_2\ar[d]^{c_1}\ar[l]_{d_1}\\
TA_0 & A_1\ar[l]^{d_0}\\
}$
\hfil
$\xymatrix{TA_2\ar[d]_{Tc_1} & A_3\ar[d]^{c_2}\ar[l]_{d_2}\\
TA_1 & A_2\ar[l]^{d_1}\\
}$	\end{center}
which also define the projections $d_{\nu+1},c_{\nu+1}$ satisfying $Tc_{\nu}\cdot d_{\nu+1}=d_{\nu}\cdot c_{\nu+1}\;(\nu=0,1)$. They induce the identity insertions $i_1=Ti\times 1_{A_1}$ and $i_2=\eta_{A_1}\times i$ rendering the diagrams
\begin{center}
$\xymatrix{&TA_{0}\ar@{=}[dd]\ar[ld]_{Ti}&&A_{1}\ar[ll]_{d_{0}}\ar[ld]_{i_{1}}\ar@{=}[dd]\\ TA_{1}\ar[dd]_{Tc_{_0}}&&A_{2}\ar[ll]_{\hs{12}d_{1}}\ar[dd]^{ c_{1}}&\\ &TA_{0}\ar@{=}[ld]&&A_{1}\ar[ll]_{\hs{-8}d_{0}}\ar@{=}[ld]\\ TA_{0}&&A_{1}\ar[ll]_{ d_{0}}&} \qquad
\xymatrix{&A_{1}\ar@{->}[dd]_{\hs{-6}c_{0}}\ar[ld]_{\eta_{A_{1}}}&&A_{1}\ar@{=}[ll]\ar[ld]_{i_{2}}\ar@{->}[dd]^{c_{0}}\\ TA_{1}\ar[dd]_{Tc_{_0}}&&A_{2}\ar[ll]_{\hs{12}d_{1}}\ar[dd]^{ c_{1}}&\\ &A_{0}\ar@{->}[ld]_{\eta_{A_{0}}}&&A_{0}\ar@{=}[ll]\ar@{->}[ld]^{i}\\ TA_{0}&&A_{1}\ar[ll]_{ d_{0}}&} $
\end{center}
commutative, so that
\begin{equation}
d_1\cdot i_1=Ti\cdot d_0,\quad c_1\cdot i_1=1_{A_1},\quad d_1\cdot i_2=\eta_{A_1},\quad c_1\cdot i_2=i\cdot c_0.
\end{equation}
With the $\mathbb T$-graph $$D_2A:=(A_0,\,A_2,\,\mu_{A_0}\cdot Td_0\cdot d_1,\, c_0\cdot c_1)$$ one obtains the morphisms
$$(1_{A_0},i_{\nu}):A\longrightarrow D_2A\quad(\nu=1,2)$$
of $\mathbb T$-graphs. In fact, one has the
derived equality $i_1\cdot i=i_2\cdot i$, and this composite morphism makes the $\mathbb T$-graph $D_2A$ pointed since the diagram
\begin{center}
	$\xymatrix{&A_{0}\ar[ld]_{\eta_{A_0}}\ar[dd]_{i_1\cdot i=}^{i_2\cdot i}\ar@{=}[rd]&\\
	TA_{0}\ar@{=}[dd]&&A_{0}\ar@{=}[dd]\\
	& A_2\ar[ld]^{\mu_{A_0}\cdot Td_0\cdot d_1}\ar[rd]^{c_{0}}	 & \\
	TA_{0}&&A_{0}\\
	 }$
\end{center}
commutes; furthermore, the morphisms $(1_{A_0}, i_{\nu})$ live in ${\sf Gph}_{\bullet}(\mathbb T)$.
\item[4.]
A {\em $\mathbb T$-category} $(A,i^A,m^A)$ is a pointed ${\mathbb T}$-graph $(A,i^A)$ which comes with a morphism $$(1_{A_0},m^A):D_2A\longrightarrow A$$ of $\mathbb T$-graphs so that the {\em composition morphism} $m=m^A:A_2\longrightarrow A_1$ satisfies the neutrality and associativity laws listed as (4) and (7) below. Asking $(1,m)$ to be a morphism of graphs amounts to asking $m$ to satisfy the equalities
\begin{equation}
d_0\cdot m=\mu_{A_{0}}\cdot Td_{0}\cdot d_{1} \quad\text{and}\quad c_0\cdot m= c_0\cdot c_1.
\end{equation}

 \begin{center}
	$\xymatrix{&A_{2}\ar[ld]_{d_{1}}\ar[dd]_{m}\ar[rd]^{c_{1}}&\\
	TA_{1}\ar[dd]_{\mu_{A_0}\cdot Td_{0}}&&A_{0}\ar[dd]^{c_{0}}\\
	& A_1\ar[ld]_{d_{0}}\ar[rd]^{c_{0}}	 & \\
	TA_{0}&&A _{0}\\
	 }$
\end{center}
The {\em neutrality law }says that $(1_{A_0},m)$ must be a common retraction to $(1_{A_0}, i_1)$ and $(1_{A_0},i_2)$ in ${\sf Gra}(\mathbb T)$, that is: one must have
\begin{equation}
m\cdot i_{1}=1_{A_{1}}=m\cdot i_{2}.
\end{equation}
in $\C$.
Note that, in fact, $(1_{A_0},m)$ lives, like the morphisms $(1_{A_0},i_{\nu})$, in ${\sf Gph}_{\bullet}(\mathbb T)$.

 %\begin{equation}
%m\cdot i_{1}=1_{A_{1}}=m\cdot i_{2}\quad\text{and}\quad m\cdot m_1= m\cdot m_2;
%\end{equation}
To formulate the associativity law, one forms the morphisms
$m_1=Tm\times c_1$ and $ m_2=(\mu_{A_1}\cdot Td_1)\times m$, as uniquely defined by the commutative diagrams
\begin{center}

 	$\hs{-7}\xymatrix{&TA_{2}\ar@{->}[dd]_{Tc_{1}}\ar[ld]_{Tm}&&A_{3}\ar[ll]_{d_{2}}\ar[ld]_{m_{1}}\ar@{->}[dd]^{c_{2}}\\
	TA_{1}\ar[dd]_{Tc_{_0}}&&A_{2}\ar[ll]_{\hs{12}d_{1}}\ar[dd]^{ c_{1}}&\\
	 &TA_{1}\ar@{->}[ld]_{Tc_{0}}&&A_{2}\ar[ll]_{d_{1}\quad\qquad}\ar@{->}[ld]^{c_{1}}\\
	TA_{0}&&A_{1}\ar[ll]_{ d_{0}}&}$\qquad\quad
%	\xymatrix{&&TTA_{1}\ar@{->}[dd]_{TTc_{0}}\ar[lld]_{\mu_{A_1}}&&TA_{2} \ar[ll]_{Td_{1}}\ar[dd]^{Tc_{1}}&A_{3}\ar[l]_{d_{2}}\ar[lld]_{m_{2}\qquad}\ar@{->}[dd]^{c_{2}}\\ TA_{1}\ar[dd]_{Tc_{_0}}&&&A_{2}\ar[lll]_{\qquad\qquad\qquad\qquad d_{1}}\ar[dd]^{ c_{1}}&&\\
%	 &&TTA_{0}\ar@{->}[lld]_{\mu_{A_0}}&&TA_{1}\ar[ll]_{Td_{0}\qquad\qquad}&A_{2}\ar[l]_{d_{1}}\ar@{->}[lld]^{m}\\ TA_{0}&&&A_{1}\ar[lll]_{ d_{0}}&&}$
$\hs{-7}\xymatrix{&TA_{2}\ar@{->}[dd]_{Tc_{1}}\ar[ld]_{\mu_{A_1}\cdot Td_1}&&A_{3}\ar[ll]_{d_{2}}\ar[ld]_{m_{2}}\ar@{->}[dd]^{c_{2}}\\
	TA_{1}\ar[dd]_{Tc_{_0}}&&A_{2}\ar[ll]_{\hs{12}d_{1}}\ar[dd]^{ c_{1}}&\\
	 &TA_{1}\ar@{->}[ld]^{\mu_{A_0}\cdot Td_{0}}&&A_{2}\ar[ll]_{d_{1}\qquad\quad}\ar@{->}[ld]^{m}\\
	TA_{0}&&A_{1}\ar[ll]_{\qquad\qquad d_{0}}&}$
\end{center}
Hence, their defining conditions are
\begin{equation}
d_1\cdot m_1=Tm\cdot d_2,\,c_1\cdot m_1=c_1\cdot c_2\quad\text{and}\quad d_1\cdot m_2=\mu_{A_1}\cdot Td_1\cdot d_2,\, c_1\cdot m_2=m\cdot c_2\,.	
\end{equation}
%Asking $(1,m)$ to be a morphism of graphs amounts to asking $m$ to satisfy the equalities
%\begin{equation}
% d_0\cdot m=\mu_{A_{0}}\cdot Td_{0}\cdot d_{1} \quad\text{and}\quad c_0\cdot m= c_0\cdot c_1.
% \end{equation}
%The neutrality laws mean equivalently that $(1_{A_0},m)$ is a common retraction to $(1_{A_0}, i_1)$ and $(1_{A_0},i_2)$ in ${\sf Gra}(\mathbb T)$ and, in fact, in ${\sf Gra}_{\bullet}(\mathbb T)$.
Like for $i_1,i_2$, the morphisms $m_1, m_2$ may be seen as belonging to the morphisms
$$(1_{A_0}, m_{\nu}):D_3A\longrightarrow D_2A\quad(\nu=1,2)$$
of $\mathbb T$-graphs, with
$$ D_3A:= (A_0,\,A_3,\,\mu_{A_0}\cdot\mu_{TA_0}\cdot TTd_0\cdot Td_1\cdot d_2,\, c_0\cdot c_1\cdot c_2).$$
We note in passing that $D_3A$ is, like $D_2A$, pointed, by the morphism $\overline{i}: A_0\longrightarrow A_3$ that is determined by the conditions
\begin{equation}
	d_2\cdot\overline{i}=\eta_{A_2}\cdot i_1\cdot i\quad\text{and}\quad c_2\cdot\overline{i}=i_1\cdot i\; (= i_2\cdot i)\,.
\end{equation}
It makes the morphisms $(1_{A_0}, m_1),\,(1_{A_0},m_2)$ live in ${\sf Gph}_{\bullet}(\mathbb T)$. The {\em associativity law} says that these morphisms are invariant under post-composition with $(1_{A_0},m)$, which simply means
\begin{equation}
m\cdot m_1	=m\cdot m_2\;.
\end{equation}
Briefly then, a $\mathbb T$-category $(A,i,m)$ must satisfy conditions (1--7) (of which (6) may be considered redundant).

 We usually write just $A$ for a $\mathbb T$-category $(A,i^A,m^A)$.
\item[5.] A morphism $f:A\longrightarrow B$ of $\mathbb T$-categories, or a {\em $\mathbb T$-functor}, is a morphism of pointed ${\mathbb T}$-graphs preserving the composition, as in (9) below. To specify this preservation condition one considers the morphism $f_2=Tf_1\times f_1$, and for later use forms also $f_3=Tf_2\times f_2$, determined by the commutative diagrams
\begin{center}
	$\xymatrix{&TA_{1}\ar@{->}[dd]_{Tc_{0}}\ar[ld]_{Tf_1}&&A_{2}\ar[ll]_{d_{1}}\ar[ld]_{f_{2}}\ar@{->}[dd]^{c_{1}}\\ TB_{1}\ar[dd]_{Tc_{_0}}&&B_{2}\ar[ll]_{\hs{8}d_{1}}\ar[dd]^{ c_{1}}&\\ &TA_{0}\ar@{->}[ld]_{Tf_{0}}&&A_{1}\ar[ll]_{\hs{-6}d_{0}}\ar@{->}[ld]^{f_{1}}\\ TB_{0}&&B_{1}\ar[ll]_{ d_{0}}&}\qquad \xymatrix{&TA_{2}\ar@{->}[dd]_{Tc_{1}}\ar[ld]_{Tf_2}&&A_{3}\ar[ll]_{d_{2}}\ar[ld]_{f_{3}}\ar@{->}[dd]^{c_{2}}\\ TB_{2}\ar[dd]_{Tc_{1}}&&B_{3}\ar[ll]_{\hs{8}d_{2}}\ar[dd]^{ c_{2}}&\\ &TA_{1}\ar@{->}[ld]_{Tf_{1}}&&A_{2}\ar[ll]_{\hs{-6}d_{1}}\ar@{->}[ld]^{f_{2}}\\ TB_{1}&&B_{2}\ar[ll]_{ d_{1}}&} $
\end{center}
so that, together with the condition that $f$ be a morphism of $\mathbb T$-graphs, one has
\begin{equation}
	d_{\nu}^B\cdot f_{\nu+1}=Tf_{\nu}\cdot d_{\nu}^A,\;c_{\nu}^B\cdot f_{\nu+1}=f_{\nu}\cdot c_{\nu}^A\quad(\nu=0,1,2).
\end{equation}
We note in passing that also for $\nu=1,2$ one obtains morphisms $$D_{\nu}f:=(f_0,f_{\nu}):D_{\nu}A\longrightarrow D_{\nu}B$$ of pointed $\mathbb T$-graphs. In addition to the preservation of the pointing, as shown by the commutative square on the left below, the preservation of the composition says that the square on the right must commute as well:
\begin{center}
	$\xymatrix{D_0A\ar[r]^{D_0f}%\ar@{}[rd]|{\circlearrowright}
	\ar[d]_{(1_{_0},i^A)}&D_0B\ar[d]^{(1_{B_0},i^B)}\\ A\ar[r]_{f}&B}$
	\hfil
	$\xymatrix{D_2A\ar[r]^{D_2f}\ar[d]_{(1_{A_0},m^A)} & D_2B\ar[d]^{(1_{B_0},m^B)}\\
	A\ar[r]_f & B\\
	}$
\end{center}
This simply means that we must have the equalities
\begin{equation}
f_1\cdot i^A=i^B\cdot f_0\quad\text{and}\quad f_1\cdot m^A=m^B\cdot f_2.	
\end{equation}
For later use we note that these conditions imply the identities
\begin{equation}
f_2\cdot i_1^A=i_1^B\cdot f_1,\;f_2\cdot i_2^A=	i_2^B\cdot f_1,\;f_2\cdot m_1^A=m_1^B\cdot f_3,\;f_2\cdot m_2^A=m_2^B\cdot f_3\;.
\end{equation}

 \medskip

 With the composition of morphisms of $\mathbb T$-graphs we obtain the category
$${\sf Cat}(\mathbb T)$$
of $\mathbb T$-categories and $\mathbb T$-functors.
\item[6.] A $\mathbb T$-category $A$ is {\em ordered} if its span $(d_0,c_0)$ is monic, so that
$d_0\cdot x=d_0\cdot y$ and $c_0\cdot x=c_0\cdot y$ implies $x=y$, for all parallel morphisms $x,y$ with codomain $A_1$.
%[[These should be pointed graph structure(s) on $A_3 $ that make $m_1, m_2$ belong to
%morphisms $NA\to MA$ of pointed $\mathbb T$-graphs, similarly as for $i_1,1_2$??]]
Trivially then, by (1) and (3), a $\mathbb T$-category structure $i, m$ of $A$ is uniquely determined by its $\mathbb T$-graph structure $d_0,c_0$. Hence, for a $\mathbb T$-graph $A$ with $(d_0,c_0)$ monic to become a $\mathbb T$-category is merely a property, with no choice of how to add the $\mathbb T$-category structure. Likewise, for a $\mathbb T$-functor $f:A\longrightarrow B$, by (8) (for $\nu=0$), the object part $f_0$ determines the morphism part $f_1$ whenever the $\mathbb T$-category $B$ is ordered.
We denote by
$${\sf Ord}(\mathbb T)$$
the category of ordered $\mathbb T$-categories and their $\mathbb T$-functors.
\end{itemize}

\begin{proposition}\label{chain}
The forgetful functor of the Eilenberg-Moore category of $\mathbb T$ to $\C$ decomposes as
\begin{center}
$\xymatrix{{\sf Alg}({\mathbb T})\ar@{^(->}[r] & {\sf Ord}(\mathbb T)\ar@{^(->}[r] & {\sf Cat}(\mathbb T)\ar[r] & {\sf Gph}_{\bullet}(\mathbb T)\ar[r] & {\sf Gph}(\mathbb T)\ar[r]^{\quad(-)_0} & \C\;,
}$	
\end{center}
where $\hookrightarrow$	denotes a full embedding, and where all functors but $(-)_0$ are faithful.
\end{proposition}
\begin{proof}
Only the full embedding of the Eilenberg-Moore category of $\mathbb T$ into ${\sf Ord}(\mathbb T)$ needs specification: it considers a $\mathbb T$-algebra $(A,a)$ as a $\mathbb T$-graph $(A,TA, 1_{TA}, a)$ and provides it with the $\mathbb T$-category structure $i=\eta_A,\,m=\mu_A$. A ${\mathbb T}$-homomorphism $f:	(A,a)\longrightarrow (B,b)$ then becomes a $\mathbb T$-functor $(f,Tf)$.
\end{proof}

 \begin{remark}
In the case that $\mathbb T$ is the identity monad on $\C$, $\mathbb T$-categories are precisely monoids in the bicategory of spans in $\C$, ${\sf Span}(\mathbb T)$. But then, unfortunately, $\mathbb T$-functors are not captured by their morphisms. That is why \cite{DawsonParePronk2010} proposed (still in the case of the identity monad) that one should set up a double category $\mathbb S{\sf pan}(\mathbb T)$ and consider $\mathbb T$-categories as lax morphisms $\mathbf 1\longrightarrow \mathbb S{\sf pan}(\mathbb T)$ defined on the terminal double category; then the vertical transformations of these lax morphisms are precisely $\mathbb T$-functors.

 In fact, for arbitrary $\mathbb T$, the bicategory ${\sf Span}(\mathbb T)$ and the presentation of $\mathbb T$-categories as their monoids appears already in \cite{Burroni1971}, and it is quite easy to establish also the double category $\mathbb S{\sf pan}(\mathbb T)$ and the above interpretation of $\mathbb T$-functors for any monad $\mathbb T$. Parts of the emerging structure are used extensively in \cite{Zawadowski2011} who, however, works in a fibrational, rather than a double-categorical, setting, as is suggested by Proposition \ref{fibred} below.
\end{remark}

 %[TO BE CHECKED: When is ${\sf Cat}(\mathbb T)$ monadic over ${\sf Gph}(\mathbb T)$?]

 \section{Special cases, examples, and remarks}
The groups of examples listed under 1--3 below, which appear already in \cite{Burroni1971}, provide a strong motivation for his work. Under items 4 and 5 below we compare this paper's $\mathbb T$-category setting with the $(\mathbb T,\sf V)$-categories as introduced in \cite{ClementinoTholen2003}, emphasizing that the latter structures are known to generalize the former ones only when $\mathbb T$ is a $\bf Set$-monad and the monoidal-closed category $\sf V$ is the 2-chain. An even broader environment to house these and many other structures is given in \cite{CruttwellShulman2010}.
\begin{itemize}
\item[1.] {\em Internal categories}. When $\mathbb T=\mathbb I\rm d_{\C}$ is the identity monad on $\C$, a $\mathbb T$-category is simply an {\em internal category} of $\C$, and a $\mathbb T$-functor is an {\em internal functor}. Ordered $\mathbb T$-categories are also known as {\em internal preorders} of $\C$. In particular, for $\C={\bf Set}$, the category ${\sf Cat}({\mathbb I}{\rm d}_{\C})$ is the ordinary category $\bf Cat$ of small categories, and ${\sf Ord}({\mathbb I}{\rm d}_{\C})$ is the category ${\bf Ord}$ of (pre)ordered sets and their monotone functions. For general $\C$, as prominent examples we mention that crossed modules are internal categories of the category of groups, and that (small strict) double categories are internal categories of ${\bf Cat}$. For further examples, we refer to the extensive literature on internal category theory that was introduced early on by the Ehresmann school (see in particular \cite {BastianiEhresmann1969} and the references given in there) and developed further for its applications in topos theory (see in particular \cite{Johnstone2002}).

 \item[2.] {\em Multicategories}. For the free-monoid (or {\em list}) monad $\mathbb L$ of $\bf Set$, an $\mathbb L$-category is a small {\em multicategory} in the sense of \cite{Lambek1969}, and it is ordered when it is a multi-ordered set, as considered in \cite{ClementinoTholen2003}. We write $\bf MulCat$ and $\bf MulOrd$ for the categories $\sf Cat(\mathbb L)$ and $\sf Ord(\mathbb L)$, respectively.

	\item[3.] {\em Relational $\mathbb T$-algebras}. If $\mathbb T$ is an arbitrary monad on $\C=\bf Set$, ordered $\mathbb T$-categories are equivalently described as Barr's {\em relational $\mathbb T$-algebras} \cite{Barr1970}, as follows. For $A\in{\sf Ord}(\mathbb T)$, we can assume $A_1\subseteqq TA_0\times A_0$, with projections
$d_0,\, c_0$. Then $A_2$ may be taken to be the subset of $TA_1\times A_0$ containing all pairs $({\mathfrak a},x)$ with $((Tc_0)({\mathfrak a}),x)\in A_1$, and the existence of a $\mathbb T$-category structure $i^A, m^A$ now reads as a reflexivity and a (hidden) transitivity condition:
\begin{itemize}
\item[($\tilde{\rm R}$)] for all $x\in A_0$: \; $(\eta_{A_0}(x),x)\in A_1$;
\item[($\tilde{\rm T}$)] for all $\mathfrak a\in TA_1,\, z\in A_0$:\; if $((Tc_0)(\mathfrak a), z)\in A_1$, then $(\mu_{A_0}((Td_0)(\mathfrak a)), z)\in A_1$.
\end{itemize}
In order to make the transitivity in ($\tilde{\rm T}$) more apparent, with
\begin{itemize}
\item[($*$)]$\mathfrak y\to z:\iff(\mathfrak y, z)\in A_1,$
\item[($**$)]$\mathfrak X\hat{\to} \mathfrak y:\iff\exists \mathfrak a\in TA_1: (Td_0)(\mathfrak a)=\mathfrak X\text{ and }(Tc_0)(\mathfrak a)=\mathfrak y,$
\end{itemize}
we may transcribe ($\tilde{\rm R}$) and ($\tilde{\rm T}$) equivalently as
\begin{itemize}
\item[(R)] for all $x\in A_0$: \; $\eta_{A_0}(x)\to x;$
\item[(T)] for all $\mathfrak X\in TTA_0,\, \mathfrak y\in TA_0, z\in A_0:$ if $\mathfrak X\hat{\to}\mathfrak y$ and $\mathfrak y\to z$, then $\mu_{A_0}(\mathfrak X)\to z$.
\end{itemize}
Indeed, assuming ($\tilde{\rm T}$), given $\mathfrak X\to\mathfrak y,\;\mathfrak y\to z$ as in (T) we obtain $\mathfrak a$ as in $(**)$, which then satisfies the hypothesis of ($\tilde{\rm T}$) and gives $\mu_{A_0}(\mathfrak X)\to z$. Hence, $(\tilde{\rm T})$ implies (T), and the converse implication follows similarly.

 In this notation, a $\mathbb T$-functor $f:A\to B$ is simply a map $f_0:A_0\to B_0$ satisfying the monotonicity condition

 (M)\; for all $\mathfrak x\in TA_0,\,y\in A_0$\,:\, if $\mathfrak x\to y$, then $(Tf_0)(\mathfrak x)\to f_0(y)$.

 In this way one sees that Burroni's category ${\sf Ord}(\mathbb T)$ is equivalent to Barr's category of relational algebras $(A_0,\to)$, with the relation $\to$ from $TA_0$ to $A_0$ satisfying the conditions (R) and (T). Indeed, if one writes the relation $\to$ as an arrow $a:TA_0\longrightarrow A_0$ in the 2-category $\bf Rel$ of sets and relations (with relational composition, functions considered as relations via their graphs, and 2-cells given by inclusion), then the conditions (R), (T), (M) are equivalently expressed by the lax-commutative diagrams defining lax Eilenberg-Moore algebras:

 \begin{center}
$\xymatrix{A_0\ar[r]^{\eta_{A_0}}\ar[rd]_{1_{A_0}}& TA_0\ar[d]^{a\quad\;\;\leq}_{\leq\;} & TTA_0\ar[l]_{\hat{T}a}\ar[d]^{\mu_{A_0}} && TA_0
\ar[r]^{Tf_0}\ar[d]_a^{\quad\;\;\leq} & TB_0\ar[d]^b\\
& A_0& TA_0\ar[l]^a && A_0\ar[r]_{f_0} & B_0\\
}$	
\end{center}
Here $\hat{T}a$ is the arrow notation for the relation $\hat{\to}$ defined by $(**)$, and $b$ is the relational arrow given by $B_1$.

 Of primary interest is the case when $\mathbb T$ is the ultrafilter monad $\mathbb U$ on $\bf Set$, whose endofunctor $U=\beta$ assigns to a set $X$ the underlying set of the Stone-$\check{C}$ech-compactification of the discrete space $X$. Then the relation $\to$ is to be read as ``converges to'', and one obtains that ${\sf Alg}(\mathbb U)\cong {\bf KHaus}$ is (isomorphic to) the category of compact Hausdorff spaces \cite{Manes1969}, and ${\sf Ord}(\mathbb U)\simeq{\bf Top}$ is (equivalent to) the category of all topological spaces \cite{Barr1970}; see \cite{HST} for an elaborate proof.

 \item[4.] {\em $(\mathbb T, \mathsf V)$-categories}. For $\C ={\bf Set}$, in {\em monoidal topology} (as presented in \cite{HST}) one generalizes Barr's relational algebras and, thus, Burroni's ordered $\mathbb T$-categories, as follows. A relation from a set $X$ to a set $Y$ is equivalently described by a map $X\times Y\to{\sf 2}$ to the ordered set ${\sf 2}=\{0<1\}$. For the notion of (small) $({\mathbb T}, {\sf V})$-{\em category} one replaces ${\sf 2}$ and its Boolean operation $\wedge$ by a quantale $\sf V$, forms the category ${\sf V}$-$\bf Rel$ of sets and $\sf V$-valued relations, and assumes that the ${\bf Set}$-monad $\mathbb T$ admits a specified {\em lax extension} $\hat{\mathbb T}$ to ${\sf V}$-$\bf Rel$. A $(\mathbb T,\sf V)${\em -category} is then nothing but a lax $\hat{\mathbb T}$-algebra in $\sf V$-$\bf Rel$, given by the left part of the above diagram; its right part describes a $(\mathbb T,\sf V)$-{\em functor}. This then defines the category
$$(\mathbb T,\sf V)\text{-}\bf Cat$$
which, for $\sf V=\sf 2$ and $\mathbb T$ extended to $\sf V$-$\bf Rel$ via $(**)$ of Example 3, reproduces Barr's category of relational $\mathbb T$-algebras. For $\sf V=\sf 1$, the terminal quantale, $(\mathbb T,\sf V)\text{-}\bf Cat$ just reproduces the category $\bf Set$.

 For quantales $\sf V$ other than $\sf 1$ or $\sf 2$, monoidal topology diverges from the realm of categories captured by Burroni's setting, even when $\mathbb T$ is the identity monad (identically extended to $\sf V\text{-}{\bf Rel}$). Indeed, in that case, $(\mathbb T,\sf V)\text{-}\bf Cat$ is just the category $\sf V$-$\sf Cat$ of small categories enriched in $\sf V$. For $\sf V$ the extended real half-line $[0,\infty]$, ordered by the natural $\geq$ and structured by $+$ as its tensor product, $\sf V$-$\bf Cat$ is Lawvere's category $\bf Met$ of (generalized) metric spaces \cite{Lawvere1973}. For the same quantale, but with $\mathbb T$ the ultrafilter monad and a suitable extension to $\sf V$-$\bf Rel$, one obtains Lowen's category $\bf App$ of approach spaces \cite{Lowen1997}, as first proved by \cite{ClementinoHofmann2002}; for an elaborate proof, see also \cite{HST}.

 A principal result of \cite{HST} says that, under mild hypotheses on the parameters $\mathbb T$ and $\sf V$, which are satisfied in all the cases mentioned thus far, the category $(\mathbb T,\sf V)\text{-}\bf Cat$ may be equivalently presented as $(\Pi,\sf 2)$-$\sf Cat$, for some monad $\Pi$ of $\bf Set$ which encodes both parameters, $\mathbb T$ and $\sf V$. However, the lax extension of $\Pi$ to be considered here is not of the type described by $(**)$ of 3. and may therefore {\em not} be assumed to be covered by the Barr-Burroni setting.

 \item[5.] For the ultrafilter monad $\mathbb U$ of $\bf Set$, we {\em conjecture} that $\mathbb U$-categories are precisely small {\em ultracategories}, as considered in \cite{ClementinoTholen2003}, but have not been able yet to complete a proof. Hence, here we must formulate as an {\bf open problem} the question whether ${\sf Cat}(\mathbb U)$ is equivalent to the category ${\bf UltCat}$.
\end{itemize}

 \section{When the object-of-objects functor is small-topological}

 In this section we strengthen Burroni's statement that, with the exception of ${\sf Alg}({\mathbb T})$, the categories occurring in Proposition \ref{chain} are all fibred over $\C$. In fact, we show that their object functors are not only fibrations, but are {\em small-topological} (in the sense of \cite{Hoffmann1972, GiuliTholen2007}; see Definition \ref{small-topological} below), provided that $\C$ has all small limits. To make our generalization more transparent, we start by sketching the proof of Burroni's original result.

 \begin{proposition}\label{fibred}{\em \cite{Burroni1971}} The object functor $$(-)_0:\A\longrightarrow \C$$ is a Grothendieck fibration, for $\A$ any of the categories of the chain
\begin{center}
$\xymatrix{{\sf Ord}(\mathbb T)\ar@{^(->}[r] & {\sf Cat}(\mathbb T)\ar[r] & {\sf Gph}_{\bullet}(\mathbb T)\ar[r] & {\sf Gph}(\mathbb T)\\
}$. 	
\end{center}
Each link preserves $(-)_0$-cartesian morphisms. The functor $(-)_0:{\sf Gph}(\mathbb T)\longrightarrow \C$ is even a bifibration.\\
\end{proposition}

 \begin{proof}
Considering the case $\A={\sf Gph}(\mathbb T)$, for a $\mathbb T$-graph $C$ and a morphism $\alpha_0:A_0\longrightarrow C_0$	 in $\C$ one constructs a $(-)_0$-cartesian lifting $\alpha=(\alpha_0,\alpha_1):A=(A_0, A_1, d_0^A, c_0^A)\longrightarrow C$ by forming the limit $A_1$ of the diagram
\begin{center}
$\xymatrix{TA_0\ar[d]_{T\alpha_0} & & A_0\ar[d]^{\alpha_0}\\
TC_0 & C_1\ar[l]_{d_o^C}\ar[r]^{c_0^C} & C_0\\
}$	
\end{center}
in $\C$, with limit projections $d_0^A:A_1\longrightarrow TA_0, \, \alpha_1:A_1\longrightarrow C_1, \, c_0^A:A_1\longrightarrow A_0$. Hence, $A_1$ may be constructed with the help of three pullbacks, as
$$A_1 \cong (TA_0\times_{TC_0}C_1)\times_{C_1}(C_1\times_{C_0}A_0);$$
equivalently, when $\C$ has binary products, $\alpha_1$ may be taken to be the pullback of $T\alpha_0\times\alpha_0$ along $\langle d_0^C,c_0^C\rangle:C_1\longrightarrow TC_0\times C_0$.
Clearly then, $\alpha:A\longrightarrow C$ is a morphism of $\mathbb T$-graphs, and its $(-)_0$-cartesianess may be confirmed easily.

 When $C$ is pointed, also $A$ is pointed, via the unique $\C$-morphism $i^A:A_0\longrightarrow A_1$ that makes $\alpha$ a morphism of pointed $\mathbb T$-graphs:
\begin{center}
$\xymatrix{
& TA_0\ar[dd] & & A_0\ar[ll]_{\eta_{A_0}}\ar@{=}[rr]\ar[dd]\ar@{-->}[ld]_{i^A} & & A_0\ar[dd]^{\alpha_0}\\
TA_0\ar[dd]_{T\alpha_0}\ar@{=}[ru] && A_1\ar[ll]_{d_0^A\qquad}\ar[rr]^{\qquad c_0^A}\ar[dd]_{\alpha_1} && A_0\ar@{=}[ru]\ar[dd] &\\
& TC_0 && C_0\ar[ll]_{\quad\qquad \eta_{C_0}}\ar@{=}[rr]\ar[ld]^{i^C} && C_0\\
TC_0\ar@{=}[ru] && C_1\ar[ll]^{d_0^C}\ar[rr]_{c_0^C} && C_0\ar@{=}[ru] &\\
}$	
\end{center}

 Similarly, when $C$ is a ${\mathbb T}$-category, there is a unique $\C$-morphism $m^A:A_2\longrightarrow A_1$ rendering the diagram

 \begin{center}
$\xymatrix{
& TA_1\ar[ld]_{\mu_{A_0}\cdot Td_0^A}\ar[dd]^{T\alpha_1} & & A_2\ar[ll]_{d_1^A}\ar[rr]^{c_1^A}\ar[dd]_{\alpha_2}\ar@{-->}[ld]_{m^A} & & A_1\ar[dd]^{\alpha_1}\ar[ld]_{c_0^A}\\
TA_0\ar[dd]_{T\alpha_0} && A_1\ar[ll]_{d_0^A\qquad}\ar[rr]^{\qquad c_0^A}\ar[dd]_{\alpha_1} && A_0\ar[dd]^{\alpha_o} &\\
& TC_1\ar[ld]^{\mu_{C_0}\cdot Td_0} && C_2\ar[ll]_{\quad\qquad d_1^C}\ar[rr]^{c_1^C\qquad\quad}\ar[ld]^{m^C} && C_1\ar[ld]^{c_0^C}\\
TC_0 && C_1\ar[ll]^{d_0^C}\ar[rr]_{c_0^C} && C_0 &\\
}$	
\end{center}
commutative. A quite elaborate diagram chase confirms that the neutrality and associativity laws (4) and (7) hold for $(A,i^A,m^A)$, and that, by design, $m^A$ is unique with the property of satisfying (3) and making $\alpha:A\longrightarrow C$ a $\mathbb T$-functor.

 Finally, when $C$ is ordered, so that the pair $(d_0^C,c_0^C)$ is monic, also $(d_0^A, c_0^A)$ is monic, whence $A$ is ordered as well.

 The $(-)_0$-cocartesian liftings in case $\A={\sf Gph}(\mathbb T)$ are trivial since one may use the same object of morphisms for the codomain of the lifting as for its given domain.
\end{proof}

 \begin{remark} For any of the categories $\A$ of Proposition \ref{fibred}, a morphism $f:A\longrightarrow B$ is $(-)_0$-cartesian if, and only if, the diagram

 \begin{equation}\label{f.f 2 pb sq}
	\xymatrix{&A_{1}\ar@{->}[ld]_{d^{A}_{0}}\ar@{->}[dd]_{f_{1}}\ar@{->}[rd]^{c^{A}_{0}}&\\ TA_{0}\ar[dd]_{Tf_{0}}&&A_{0}\ar[dd]^{f_{0}}\\&B_{1}\ar[ld]_{d^{B}_{0}}\ar[rd]^{c^{B}_{0}}&\\ TB_{0}&&B_{0} }
\end{equation}
is a bi-pullback in $\C$; in the presence of binary products, that is, if

 \begin{equation}\label{T-f.f}
	\xymatrix{A_{1}\ar[d]_{(d^{A}_{0}, c^{A}_{0})}\ar[rr]^{f_{1}}&&B_{1}\ar[d]^{(d^{B}_{0}, c^{B}_{0})}\\ TA_{0}\times A_{0}\ar[rr]_{Tf_{0}\times f_{0}}&&TB_{0}\times B_{0}}
\end{equation}
is a pullback diagram in $\C$. The case $\A={\sf Cat}(\bf Set)={\bf Cat}$ suggests to call $(-)_0$-cartesian morphisms {\em fully faithful} also in the general case.
\end{remark}

 In any fibred category $\A$ over $\C$ with fibration $P$ one has the orthogonal factorization system $(P^{-1}({\rm Iso}\,\C),\{P\text{-cartesian}\})$. For $\A$ as in Proposition \ref{fibred}, calling a morphism $f$ a 0-{\em isomorphism} when $f_0$ is an isomorphism, we obtain in particular:

 \begin{corollary}
Any of the categories $\A$ of Proposition {\em \ref{fibred}} has an orthogonal\\ $(0\text{-isomorphism},\text{fully faithful})$-factorization system.	
\end{corollary}

 For our generalization of Proposition \ref{fibred} we use the notions of $\J$-topological functor and of wide pullback of a $\J$-indexed cocone, which we present first.

 \begin{definition}\label{small-topological}
{\em (1)} Let $P:\A\longrightarrow	\C$ be a functor, $X$ an object in $\C$ and $H:\J\longrightarrow\A$ a diagram in $\A$. Then a natural transformation $\xi:\Delta X\longrightarrow PH$ is also called a {\em $P$-structured cone} in $\C$; it is {\em indexed by $\J$}. A {\em lifting (along $P$)} of such a $P$-structured cone is a cone $\alpha:\Delta A\longrightarrow H$ in $\A$ with $PA=X$ and $P\alpha=\xi$, and $\alpha$ is $P$-{\em cartesian }if for every cone $\beta:\Delta B\longrightarrow H$ in $\A$ and every morphism $g:PB\longrightarrow PA=X$ in $\C$ with $\xi\cdot\Delta g=P\beta$ one has a unique morphism $f:B\longrightarrow A$ in $\A$ with $\alpha \cdot\Delta f=\beta$ and $Pf=g$.

 \begin{equation*}
\xymatrix@R=5ex@C=4em{
{\Delta B}\ar@{-->}[r]_-{\Delta f}\ar@/^3ex/[rr]^{\beta} &
	{\Delta A}\ar[r]_-{\alpha} & {H} &
	\A \ar[d]^{P} \\
{\Delta PB}\ar[r]_-{\Delta g}\ar@/^3ex/[rr]^{P\beta} &
	{\Delta X}\ar[r]_-{\xi} & {PH} & \C
}
\end{equation*}

 {\em (2)} For a category $\J$, the functor $P:\A\longrightarrow \C$ is {\em $\J$-topological} if every $\J$-indexed $P$-structured cone has a $P$-cartesian lifting. $P$ is {\em (finite-; small-)topological} if it is $\J$-topological for all (finite; small) categories $\J$.

 {\em (3)} $P$ is {\em $(\J\text{-; finite-; small-}$) cotopological} if $P^{\rm op}:\A^{\rm op}\longrightarrow\C^{\rm op}$ is $(\J\text{-; finite-; small-}$) topological.
\end{definition}

 \begin{remark}\label{topremark}
(1) The functor $P$ is $\emptyset$-topological if, and only if, $P$ admits a right-adjoint right-inverse functor, and $P$ is ${\sf 1}$-topological (for the terminal category ${\sf 1}$) if, and only if, $P$ is a (Grothendieck) fibration.

 (2) $P$ is topological if, and only if, $P$ is $\J$-topological for all discrete categories ${\J}$; equivalently, $P$ is $\J$-cotopological for all discrete categories $\J$. A topological functor is faithful and both, a fibration and a cofibration, such that the fibres have all (discretely indexed) limits (or colimits); and these properties characterize the topologicity of $P$ (see \cite{Tholen1979, HST} for details).
\end{remark}

 The following well-known property describes the interaction of the notions of limit and cartesianess (as stated more generally in \cite{Tholen1979}). Its proof is straightforward.

 \begin{lemma}\label{toplemma}
Let $\alpha:\Delta A\longrightarrow H$ be a cone in $\A$ transformed by $P:\A\to\X$ into a limit cone. Then $\alpha$ is a limit cone itself if, and only if, $\alpha$ is $P$-cartesian.	Consequently, if $P$ is $\J$-topological and $\C$ $\J$-complete, so is $\A$, with $\J$-indexed limits being preserved by $P$.
\end{lemma}

 Here is the second notion that we find convenient to use in generalizing Proposition \ref{fibred}. It just generalizes the standard notion of wide pullback (or intersection) of a sink of morphisms from the discrete to the non-discrete level.

 \begin{definition}
{\em (1)} A {\em	wide pullback} of a cocone $\beta:H\longrightarrow\Delta B$ with $H:\J\longrightarrow\C$ in a category $\C$ is a cone $\alpha:\Delta A\longrightarrow H$ together with a morphism $p:A\longrightarrow B$ such that $\beta\cdot\alpha=\Delta p$ and the obvious universal property holds: for every cone $\gamma:\Delta C\longrightarrow H$ and morphism $q:C\longrightarrow B$ with $\beta\cdot\gamma=\Delta q$ one has factorizations $\gamma=\alpha\cdot\Delta f$ and $q=p\cdot f$, for a unique morphism $f:C\longrightarrow A$.

 {\em (2)} We say that the category $\C$ {\em has $\J$-indexed wide pullbacks} if every $\J$-indexed cocone in $\C$ has a wide pullback.
\end{definition}

 \begin{remark}\label{widepbsremarks}
(1) Trivially, $\emptyset$-indexed and ${\sf 1}$-indexed wide pullbacks	exist in every category. A wide pullback of $\beta: H\longrightarrow \Delta B$ is just a limit of $H$ when $B=1$ is terminal in $\C$.

 (2) $\J$-indexed wide pullbacks in $\C$ may be constructed from $\J$-indexed limits in $\C$ when $\C$ has also equalizers of $\J_0$-indexed families of parallel arrows, where $\J_0$ is the class (or discrete category) of objects of $\J$. Indeed, in the notation of the Definition, just form the limit $D$ of $H$ with projections $\delta_j:D\longrightarrow Hj$ and then equalize the family of morphisms $\beta_j\cdot\delta_j,\; j\in\J_0$.

 (3) Since $\C$ is assumed to have pullbacks, note that the equalizer of a $\J_0$-indexed family of morphisms $f_j:A\longrightarrow B$ exists in $\C$ when $\C$ has $\J_0$-indexed products, as the equalizer may be constructed as the pullback of the diagonal morphism $B\longrightarrow B^{\J_0}$ along $\prod_{j\in\J_0}f_j:A^{\J_0}\longrightarrow B^{\J_0}$.
\end{remark}

We can now expand on Burroni's proposition and comprehensively state the following theorem.
	
\begin{theorem}\label{topologicity}
In addition to pullbacks, let $\C$ have binary products and $\J$-indexed wide pullbacks, for some category $\J$. Then, for $\A$ any of the categories of the chain
\begin{center}
$\xymatrix{{\sf Ord}(\mathbb T)\ar@{^(->}[r] & {\sf Cat}(\mathbb T)\ar[r] & {\sf Gph}_{\bullet}(\mathbb T)\ar[r] & {\sf Gph}(\mathbb T)\,,\\
}$ 	
\end{center}

 {\em (1) }
the object functor $(-)_0:\A\longrightarrow \C$ is $\J$-topological, has a right-adjoint right-inverse functor and, except for $\A={\sf Gph}({\mathbb T})$, also a left-adjoint right-inverse functor. Furthermore,

 {\em (2)}
every category of the chain has all $\J$-indexed limits if $\C$ has them, and they are preserved by its object functor, as well as by the link of the chain departing from it which, in fact, preserves the $(-)_0$-cartesianess of $\J$-indexed cones; and,

 {\em (3)} irrespective of the above hypotheses on $\C$, also ${\sf Alg}(\mathbb T))$ has all $\J$-indexed limits if $\C$ has them, and they are preserved by the full embedding $$\xymatrix{{\sf Alg}(\mathbb T)\ar@{^(->}[r]
& {\sf Ord}(\mathbb T)}\;.$$
\end{theorem}

 \begin{proof}
(1) In order to show the $\J$-topologicity of $(-)_0:\A\longrightarrow\C$ for the given choices of $\A$, we consider a diagram $H:\J\longrightarrow\A$ and a cone $\alpha_0:\Delta A_o\longrightarrow (-)_0H$ in $\C$. With $Hj=(C_0^j,C_1^j,d_0^j,c_0^j)$ for all $j\in\J_0$, these data give us for all $\sigma:j\longrightarrow k$ in $\J$ the diagrams
\begin{center}
$\xymatrix{ TA_0\ar[dd]_{T\alpha_o^j}\ar[dddr]^{T\alpha_0^k} & & & & A_0\ar[dd]_{\alpha_0^j}\ar[dddr]^{\alpha_0^k} &\\
&&&&&\\
TC_0^j\ar[rd]_{T((H\sigma)_0)} & & C_1^j\ar[ll]^{\quad d_0^j}\ar[rr]_{c_0^j\quad}\ar[rd]_{(H\sigma)_1} & & C_0^j\ar[rd]_{(H\sigma)_0} &\\
& TC_0^k && C_1^k\ar[ll]^{d_0^k}\ar[rr]_{c_0^k} && C_0^k
}$	
\end{center}
of which, taken collectively, we would like to form the limit in $\C$, to be called $A_1$, with limit projections $$d:A_1\longrightarrow TA_0,\quad \alpha_1^j:A_1\longrightarrow C_1^j\quad(j\in\J_0),\quad c:A_1\longrightarrow A_0.$$
If $\J=\emptyset$, that limit is simply the product $TA_0\times A_0$. Also for $\J={\sf 1}$ the existence of the limit is guaranteed, as we are then just in the setting of Proposition \ref{fibred}. In fact, for any $\J\neq\emptyset$, under our hypotheses, the construction of the limit $A_1$ may be based on the case $\J={\sf 1}$, as follows. For every $j\in \J_0$, we form the pullback $\langle\delta^j,\gamma^j\rangle:A_1^j\longrightarrow TA_0\times A_0$ of $\langle d_0^j,c_0^j\rangle:C_1^j\longrightarrow TC_0^j\times C_0^j$ along $T\alpha_0^j\times\alpha_0^j: TA_0\times A_0\longrightarrow TC_0^j\times C_0^j$, with $\pi^j$ denoting the other pullback projection. Clearly, the formation of $A_1^j$ is functorial in $j$, making the morphisms $\langle \delta^j,\gamma^j\rangle$ form a $\J$-indexed cocone, of which one may take the wide pullback, given by morphisms $\rho^j:A_1\longrightarrow A_1^j$, such that $\langle\delta^j,\gamma^j\rangle\cdot \rho^j=\langle d_0,c_0\rangle$ is constant. Finally, one puts $\alpha_1^j:=\pi^j\cdot \rho^j$ to complete the limit construction.

 Without reference to its actual construction, just using the universal property of the limit, in complete analogy to the steps taken in the proof of Proposition \ref{fibred} one can now show that, with $A=(A_0,A_1,d_0,c_0)$ and $\alpha^j=(\alpha_0^j,\alpha_1^j)$, one obtains a $(-)_0$-cartesian lifting of $\alpha:\Delta A\longrightarrow H$ of $\alpha_0$. This fact remains true also in the case $\J=\emptyset$, so that one has a right-adjoint right-inverse to $(-)_0$, by Remark \ref{topremark}(1).

 Already in Section 2 we mentioned the trivial fact that the ``discrete'' functor $D$ is left-adjoint right-inverse to $(-)_0 :{\sf Gph}_{\bullet}\longrightarrow \C$, and it is easy to see that $D$ actually takes values in ${\sf Ord}(\mathbb T)$.

 (2) From the proof of (1) we conclude that the links of the chain preserve $\J$-cartesianness of cones. The remaining statements follow directly from Lemma \ref{toplemma}.

 (3) It is well known how to construct a $\J$-indexed limit in ${\sf Alg}(\mathbb T)$ when the limit of the underlying diagram in $\C$ exists---without any other hypotheses on $\C$. It is easy to see that, after embedding the limit cone obtained in this way into ${\sf Ord}(\mathbb T)$, it actually forms a limit of the diagram displayed in (1).
\end{proof}

 \begin{corollary}\label{summarytopological}
If $\C$ is finitely complete, so are the categories ${\sf Gph}_{\bullet}(\mathbb T), {\sf Cat}(\mathbb T)$ and ${\sf Ord}(\mathbb T)$, with their $\C$-valued	object functors $(-)_0$ being finite-topological and having both, a left-adjoint right-inverse and a right-adjoint right-inverse functor. These categories are complete if $\C$ is, with $(-)_0$ being small-topological. If, in addition, $\C$ is well-powered, the object functor $(-)_0:{\sf Ord}(\mathbb T)\longrightarrow \C$ is even topological, and then ${\sf Ord}(\mathbb T)$ is also cocomplete when $\C$ is cocomplete.
\end{corollary}

 \begin{proof}
All but the last statement follow from Theorem \ref{topologicity}, in conjunction with Remark \ref{widepbsremarks}(2),(3). For the last statement, note that the wide pullback of a large family of monomorphisms may be constructed as the wide pullback of a small (representative) family when $\C$ is well-powered, so that the needed limit in the proof of Theorem \ref{topologicity} exists even when $\J$ is large (and discrete). Since, by Remark \ref{topremark}(2), the topological functor
$(-)_0:{\sf Ord}(\mathbb T)\longrightarrow \C$ is also cotopological, $(-)_0$ ``lifts" not only the existence of limits, but also of colimits.
\end{proof}

 It is important to observe that Corollary \ref{summarytopological} requires no limit preservation by the functor $T$. But the computation of finite limits in ${\sf Cat}(\mathbb T)$ (and likewise in any of the categories $\A$ of Theorem \ref{topologicity}) in terms of terminal objects and pullbacks becomes much easier when the functor $T$ preserves pullbacks, as the following corollary shows. However, in Proposition \ref{pbstability} below we give an explicit construction of pullbacks in ${\sf Cat}(\mathbb T)$ {\em without} the hypothesis of pullback preservation by $T$.

 \begin{corollary}\label{Tpreservespbs}
{\em (1)} For a terminal object $1$ in $\C$, the $\mathbb T$-graph $1=(1, T1, 1_{T1}, !)$ is terminal in ${\sf Gph}(\mathbb T)$, and the $\mathbb T$-category $(1,\eta_1,\mu_1)$ is terminal in ${\sf Cat}(\mathbb T)$.

 {\em (2)} If $T:\C\longrightarrow\C$ preserves pullbacks, then ${\sf Cat}(\mathbb T)$ has pullbacks that are preserved by $(-)_0:{\sf Cat}(\mathbb T)\longrightarrow \C$.
\end{corollary}

 \begin{proof} 
 (1) The assertion follows from the case $\J=\emptyset$ as considered in the proof of Theorem \ref{topologicity} since one trivially has the product $T1\times 1\simeq T1$ in $\C$.

 (2) Given $\mathbb T$-functors $f:A\longrightarrow C,\; g:B\longrightarrow C$, one forms the pullbacks $(g_0',f_0')$ of $(f_0, g_0)$ and $(g_1',f_1')$ of $(f_1,g_1)$ in $\C$ to obtain the solid- and dashed-arrow parts of the following diagram:	
\begin{center}
$\xymatrix{TP_0\ar[d]_{Tg_0'}\ar[rdd]^{Tf_0'} && P_1\ar@{-->}[d]_{g_1'}\ar@{-->}[rdd]^{f_1'}\ar@{..>}[ll]^{d_0^P}\ar@{..>}[rr]_{c_0^P} && P_0\ar[d]_{g_0'}\ar[rdd]^{f_0'} &\\
TA_0\ar[rdd]_{Tf_0} && A_1\ar[rdd]_{f_1}\ar[ll]^{\quad d_0^A}\ar[rr]_{\quad c_0^A} && A_0\ar[rdd]_{f_0} &\\
& TB_0\ar[d]^{Tg_0} && B_1\ar[d]^{g_1}\ar[ll]_{d_0^B\quad}\ar[rr]^{c_0^B\quad} && B_0\ar[d]^{g_0}\\
& TC_0 && C_1\ar[ll]_{d_0^C}\ar[rr]^{c_0^C} && C_0\\
}$	
\end{center}
As the left and right vertical panels are pullbacks (since $T$ preserves pullbacks), one obtains the induced morphisms $d_0^P,\,c_0^P$ completing this commutative diagram. It is now easy to verify that $P_1$, together with the dotted and dashed arrows serving as projections, is a limit of the solid-arrow part of diagram in $\C$. Therefore, the proof of Theorem \ref{topologicity} shows that the $\mathbb T$-graph $P$ defined by the above diagram can be augmented to become a $\mathbb T$-category, such that the pair $(g',f')$ serves as a pullback of $(f,g)$ in ${\sf Cat}(\mathbb T)$.
\end{proof}

\begin{examples}
(1) The object functor ${\bf Cat}\longrightarrow{\bf Set}$	
is small-topological \cite{Hoffmann1972} and has both, a fully faithful left adjoint and a fully faithful right adjoint, by an application of Corollary \ref{summarytopological} to the identity monad on $\bf Set$. But it is neither topological (since it is not even faithful), nor small-cotopological, not even a cofibration.

 By contrast, the forgetful functor ${\bf Ord}\longrightarrow{\bf Set}$ is (co)topological.

 (2) Trading the identity monad for the list monad $\mathbb L$ on $\bf Set$ one sees that, likewise, the object functor $\bf MulCat\to \bf Set$ is small-topological and has both adjoints, but is not a cofibration, whereas the forgetful functor ${\bf MulOrd}\longrightarrow\bf Set$ is (co)topological.

 (3) Considering the ultrafilter monad $\mathbb U$ of $\bf Set$, one obtains the non-trivial fact that the object functor ${\sf Cat}(\mathbb U)\longrightarrow\bf Set$ is small-topological and has both adjoints, but it is not a cofibration. When restricting it to ordered $\mathbb U$-categories, {\em i.e.}, to topological spaces, one reproduces the obvious fact that the forgetful functor $\bf Top\longrightarrow\bf Set$ is (co)topological.
\end{examples}

 \section{Factoring a T-functor universally through a perfect T-functor}

 The {\em comprehensive factorization }for functors of ordinary categories as given by \cite{StreetWalters1973} carries over to internal functors; see \cite{Johnstone2002}. Here we generalize this construction to $\mathbb T$-functors when the endofunctor $T$ of $\C$ preserves coequalizers of reflexive pairs, these having to be stable under pullback in $\C$. The construction establishes an orthogonal factorization system in ${\sf Cat}(\mathbb T)$, as we show in this and the next section.

 It is clear how to translate the relevant notion of discrete Grothendieck cofibration into Burroni's context:

 \begin{definition}
A $\mathbb T$-functor $f:A\to B$ is a {\em discrete cofibration (\cite{Grothendieck1961})} or a {\em discrete 0-fibration (\cite{Gray1969}, \cite{StreetWalters1973})}, nowadays more commonly known as a {\em discrete opfibration (\cite{Johnstone2002})}, and here, in view of the topological example below, simply referred to as a {\em perfect} $\mathbb T$-functor, if
\begin{center}
$\xymatrix{TA_0\ar[d]_{Tf_0}& A_1\ar[l]_{d_0^A}\ar[d]^{f_1}\\
TB_0&B_1\ar[l]^{d_0^B}\\
}$	
\end{center}
is a pullback diagram in $\C$. We denote by ${\sf Perf}(\mathbb T)$ the class of perfect morphisms in ${\sf Cat}(\mathbb T)$; those with fixed codomain $B$ are the objects of the full subcategory ${\sf Perf}(\mathbb T)/B$ of the slice category ${\sf Cat}(\mathbb T)/B$.
\end{definition}

 Before proceeding, let us indicate that the notion of perfect $\mathbb T$-functor leads to important types of morphisms beyond the realm of category theory.

 \begin{examples}
(1) For $\mathbb T$ the identity functor on ${\sf Set}$, we obtain the classical notion of discrete cofibration $f:A\longrightarrow B$ of small categories:	 for every object $x$ in $A$ and every morphism $\beta:fx=y\longrightarrow y'$ in $B$, one has a unique morphism $\alpha:x\longrightarrow x'$ in $A$ with $f\alpha=\beta$. This characterization extends to small multicategories, {\em i.e.}, to the case when $\mathbb T$ is the list monad on ${\sf Set}$, as graphically indicated by
\begin{center}
$\xymatrix{x\ar@{-}[d]\ar@{..>}[r]^{\alpha} & x'\ar@{..}[d] && (x_1, ..., x_n)\ar@{-}[d]\ar@{..>}[r]^{\qquad\alpha} & x'\ar@{..}[d] && A\ar[d]^f\\
y\ar[r]^{\beta} & y' && (y_1, ..., y_n)\ar[r]^{\qquad \beta} & y' && B
}$	
\end{center}

 (2) Considering now ${\sf Ord}(\mathbb T)$ instead of ${\sf Cat}(\mathbb T)$, where $\mathbb T={\rm Id}_{\sf Set}$, so that the arrows $\alpha$ and $\beta$ above are to be read as $\leq$, we see that a monotone map $f:X\longrightarrow Y$ of preordered sets is a discrete cofibration precisely when for all $y'\in Y$ one has 1. $f^{-1}(\downarrow y')	\subseteq \downarrow(f^{-1}y')$ and 2. $\forall x',x''\in f^{-1}(y'):\;(\downarrow x'\cap\downarrow x''\neq\emptyset\Rightarrow x'=x'')$. Likewise, taking now $\mathbb T$ to be the ultrafilter monad, we see that a continuous map $f:X\longrightarrow Y$ of topological spaces is a discrete cofibration precisely when for every ultrafilter $\mathfrak x$ on $X$ such that its image $f[\mathfrak x]$ under $f$ converges to $y'\in Y$, one has that $\mathfrak x$ converges in $X$ to a unique point $x'\in f^{-1}y'$.
\begin{center}
$\xymatrix{x\ar@{-}[d]\ar@{..>}[r]^{\leq} & x'\ar@{..}[d] &&& \mathfrak x\ar@{-}[d]\ar@{..>}[r]^{\rm conv} & x'\ar@{..}[d] &&& X\ar[d]^f\\
y\ar[r]^{\leq} & y' &&& \mathfrak y\ar[r]^{\rm conv} & y' &&& Y
}$	
\end{center}
Separating the existence and uniqueness assertions in this description, one readily proves that a discrete cofibration $f:X\longrightarrow Y$ in ${\sf Top}$ is characterized as being 1. {\em proper} \cite{Bourbaki1989}, so that every pullback of $f$ is a closed map or, equivalently, $f$ is a closed map with compact fibres, and 2. $f$ is {\em Hausdorff} \cite{James1989}, so that distinct point of the same fibre of $f$ may be separated by disjoint open sets in $X$. In the absence of any separation conditions on their domains and codomains, such maps have been called {\em perfect} in various generalized contexts (see, for example, \cite{Tholen1999} and \cite{HST}), which has led us to using that name also in the current generalized context of Burroni's $\mathbb T$-categories. Of course, this topological characterization returns the characterization of discrete cofibrations of ordered sets as given above when these are provided with the Alexandroff topology.
\end{examples}

 Recall that a pair of parallel morphisms in a category is {\em reflexive} if the two morphisms admit a common section in the category. Now, this paper's central result reads as follows:

 \begin{theorem}\label{reflective}
Assume that the category $\C$ with pullbacks also admits coequalizers of reflexive pairs and that these be stable under pullback and be preserved by $T$. Then,
for every $\mathbb T$-category $B$, ${\sf Perf}(\mathbb T)/B$ is reflective in ${\sf Cat}(\mathbb T)/B$.
\end{theorem}

 \begin{proof}
Given a $\mathbb T$-functor $f:A\longrightarrow B$, considered as an object of ${\sf Cat}(\mathbb T)/B$, we must find a perfect $\mathbb T$-functor $p:P\longrightarrow
B$ and a $\mathbb T$-functor $r:A\longrightarrow P$ with $f=p\cdot r$ which reflects $f$ into ${\sf Perf}(\mathbb T)/B$.

 {\em STEP 1: Defining $P_0$ and $p_0:B_0\longrightarrow P_0$\,. } \\ We start by building the back and front faces of the following commutative diagram by consecutive pullbacks in $\C$:
\begin{center}
$\xymatrix{ && TA_0\ar[d]^{Tf_0}\ar[lld]_{Ti^A}& TA_0\times_{TB_0}B_1\ar[l]_{\pi_1\quad}\ar[d]^{\pi_2}\ar@{..>}[lld]_{j\quad}\\
TA_1\ar[d]^{Tf_1}\ar@/_25pt/[dd]_{T(f_0\cdot c_0^A)} & TA_1\times_{TB_0}B_1\ar[l]^{\pi_1'\quad}\ar[d]\ar@/_25pt/[dd]^{\pi_2'} & TB_0\ar@{=}[d]\ar[lld]_{\quad\quad\quad\quad Ti^B} & B_1\ar[l]_{d_0^B}\ar@{=}[d]\ar[lld]_{i_1^B\quad}\\
TB_1 \ar[d]^{Tc_0^B}& B_2\ar[l]^{d_1^B\quad}\ar[d] & TB_0\ar@{=}[lld] & B_1\ar[l]_{d_0^B}\ar@{=}[lld]\\
TB_0 & B_1\ar[l]^{d_0^B}&&
}$	
\end{center}
here, the pullback projection $\pi_2'$ is factored as
\begin{center}$\xymatrix{TA_1\times_{TB_0}B_1\ar[rr]^{Tf_1\times 1_{B_1}\quad} & &B_2=TB_1\times_{TB_0}B_1\ar[r]^{\quad\quad\quad c_1^B} & B_1\\
}.$
\end{center}
It is then easy to see that the diagonal arrows $Ti^A$ and $i_1^B=Ti^B\times 1_{B_1} $ induce the dotted arrow $j=Ti^A\times 1_{B_1}$, determined by $\pi_1'\cdot j=Ti^A\cdot\pi_1$ and $(Tf_1\times 1_{B_1})\cdot j =i_1^B\cdot\pi_2$.

 Next, one defines the morphisms $$k=(\mu_{A_0}\cdot Td_0^A\cdot\pi_1',\;m^B\cdot (Tf_1\times 1_{B_1})),\quad\ell=(Tc_0^A\cdot\pi_1',\,c_1^B\cdot (Tf_1\times 1_{B_1}))=Tc_0^A\times 1_{B_1}$$ to render the diagram
\begin{center}
$\xymatrix{ && TA_1\ar[dd]^{Tf_1}\ar@<-0.5ex>[lld]_{\mu_{A_0}\cdot Td_0^A}\ar@<0.5ex>[lld]^{Tc_0^A} && TA_1\times_{TB_0} B_1\ar[ll]_{\pi_1'}\ar[dd]^{Tf_1\times 1_{B_1}}\ar@<-0.5ex>[ld]_{k}\ar@<0.5ex>[ld]^{\ell}\\
TA_0\ar[dd]_{Tf_0} &&& TA_0\times_{TB_0}B_1\ar[dd]_{\pi_2}\ar[lll]^{\pi_1\quad\quad\quad} &\\
& & TB_1\ar@<-0.5ex>[lld]_{\mu_{B_0}\cdot Td_0^B}\ar@<0.5ex>[lld]^{Tc_0^B} && B_2=TB_1\times_{TB_0}B_1\ar[ll]_{\quad d_1^B}\ar@<-0.5ex>[ld]_{m^B}\ar@<0.5ex>[ld]^{c_1^B}\\
TB_0 &&& B_1\ar[lll]^{d_0^B} &\\
}$	
\end{center}
commutative in the obvious sense. Now $j$ is easily seen to be a common section of the morphisms $k$ and
$\ell$.
We can therefore form the coequalizer $z:TA_0\times_{TB_0}B_1\longrightarrow P_0$ of the reflexive pair $k,\ell$. Then the solid-arrow part of the diagram
\begin{center}
$\xymatrix{TA_1\times_{TB_0}B_1\ar@<0.5ex>[r]^k\ar@<-0.5ex>[r]_{\ell}\ar[d]_{Tf_1\times 1_{B_1}} & TA_0\times_{TB_0}B_1\ar[r]^{\quad\quad z}\ar[d]^{\pi_2} & P_0\ar@{..>}[d]^{p_0}\\
B_2=TB_1\times_{TB_0}B_1\ar@<0.5ex>[r]^{\quad\quad\quad m^B}\ar@<-0.5ex>[r]_{\quad\quad\quad c_1^B} & B_1\ar[r]^{c_0^B} & B_0
}$	
\end{center}
commutes in the obvious sense and produces the morphism $p_0:P_0\longrightarrow B_0$ with $p_0\cdot z= c_0^B\cdot \pi_2$.\;\footnote{We note that, together with the sections $i^B:B_0\longrightarrow B_1$ and $i_2^B:B_1\longrightarrow B_2$, the lower row of the diagram constitutes a {\em contractible coequalizer} (as defined in \cite{BarrWells1985}); moreover, the pair $(m^B, c_1^B)$ is reflexive, as it has the common section $i_1^B:B_1\longrightarrow B_2$.}

 \medskip

 {\em STEP 2: Defining $P_1,\; p_1:P_1\longrightarrow B_1,\, d_0^P$ and $c_0^P$.}\\Aiming to make $p_0$ the object part of a perfect $\mathbb T$-functor $p:P\longrightarrow B$, one must obviously define $P_1, d_0^P$ and $p_1$ by the pullback diagram
\begin{center}
$\xymatrix{TP_0\ar[d]_{Tp_0}& P_1=TP_0\times_{TB_0} B_1\ar[l]^{d_0^P\quad\quad\quad}\ar[d]^{p_1} \\
TB_0&B_1\ar[l]_{d_0^B}\, .\\
}$	
\end{center}
However, establishing the codomain morphism $c_0^P$ turns out to be much harder. To this end, one first pulls back the $T$-image of the above coequalizer diagram along the morphism $d_0^P$. This, by hypothesis, gives us the coequalizer $\overline{z}$ of the pair $\overline{k},\overline{\ell}$, as displayed by the following commutative diagram, in which we also fix our notation for the relevant pullback projections:
\begin{center}
$\xymatrix{T(TA_1\times_{TB_0}B_1)\times_{TB_0}B_1\ar@<0.5ex>[r]^{\overline{k}}\ar@<-0.5ex>[r]_{\overline{\ell}}\ar@/^40pt/[rrr]_{\widetilde{\pi_2}'}\ar[d]_{\widetilde{\pi_1}'} & T(TA_0\times_{TB_0}B_1)\times_{TB_0}B_1\ar[r]_{\quad\quad\quad\quad \overline{z}}\ar[d]^{\widetilde{\pi_1}}\ar@/^15pt/[rr]^{\widetilde{\pi_2}} & P_1\ar[d]^{d_0^P}\ar[r]_{p_1} & B_1\ar[d]^{d_0^B}\\
T(TA_1\times_{TB_0}B_1)\ar@<0.5ex>[r]^{Tk}\ar@<-0.5ex>[r]_{T\ell}\ar@/_42pt/[rrr]^{T(c_0^B\cdot\pi_2')} & T(TA_0\times_{TB_0}B_1)\ar[r]^{\quad\quad\quad Tz}\ar@/_15pt/[rr]_{T(c_0^B\cdot \pi_2)} & TP_0\ar[r]^{Tp_0} & TB_0\\
}$	
\end{center}
Next, with the abbreviations
$$C=TA_0\times_{TB_0}B_1, \qquad D=TC\times_{TB_0}B_1,\qquad C'=TA_1\times_{TB_0}B_1,\qquad D'=TC'\times_{TB_0}B_1,$$
one has the morphisms $T\pi_2\times 1_{B_1}$ and $T\pi_2'\times 1_{B_1}$ of the commutative diagrams
\begin{center}
$\xymatrix{ & TC\ar[dd]^{T(c_0^B\cdot\pi_2)}\ar[ld]_{T\pi_2} && D\ar[ll]_{\widetilde{\pi_1}}\ar[ld]_{T\pi_2\times 1_{B_1}}\ar[dd]^{\widetilde{\pi_2}}\\
TB_1\ar[dd]_{Tc_{0}^B} && B_2\ar[ll]_{ d^B_1
\quad\quad}\ar[dd]^{c^B_1} & \\
& TB_{0}\ar@{=}[ld] && B_1\ar[ll]_{d_0^B\quad\quad\quad\quad}\ar@{=}[ld]\\
TB_{0}&& B_{1}\ar[ll]^{d^B_0} & \\
}$\quad
$\xymatrix{ & TC'\ar[dd]^{T(c_0^B\cdot\pi_2')}\ar[ld]_{T\pi_2'} && D'\ar[ll]_{\widetilde{\pi_1}'}\ar[ld]_{T\pi_2'\times 1_{B_1}}\ar[dd]^{\widetilde{\pi_2}'}\\
TB_1\ar[dd]_{Tc_{0}^B} && B_2\ar[ll]_{ d^B_1
\quad\quad}\ar[dd]^{c^B_1} & \\
& TB_{0}\ar@{=}[ld] && B_1\ar[ll]_{d_0^B\quad\quad\quad\quad}\ar@{=}[ld]\\
TB_{0}&& B_{1}\ar[ll]^{d^B_0} & \\
}$
\end{center}
Now one defines the morphisms $t,t'$, by fitting them into the commutative diagrams
\begin{center}
$\xymatrix{D\ar[rr]^{\quad T\pi_2\times 1_{B_1}}\ar[dd]_{\widetilde{\pi_1}}\ar@{..>}[rd]_t && B_2\ar[dd]_{d_1^B}\ar[rd]^{m^B} &\\
& C\ar[rr]_{\pi_2\quad\quad\quad\quad}\ar[dd]_{\pi_1} && B_1\ar[dd]_{d_0^B} \\
TC\ar[rr]^{T\pi_2\quad\quad\quad\quad\quad}\ar[rd]_{\mu_{A_0}\cdot T\pi_1} && TB_1\ar[rd]_{\mu_{B_0}\cdot Td_0^B} &\\
& TA_0\ar[rr]_{Tf_0\quad\quad\quad} && TB_0 \\
}$	\quad
$\xymatrix{D'\ar[rr]^{\quad T\pi_2'\times 1_{B_1}}\ar[dd]_{\widetilde{\pi_1'}}\ar@{..>}[rd]_{t'} && B_2\ar[dd]_{d_1^B}\ar[rd]^{m^B} &\\
& C'\ar[rr]_{\pi_2'\quad\quad\quad\quad}\ar[dd]_{\pi_1'} && B_1\ar[dd]_{d_0^B} \\
TC'\ar[rr]^{T\pi_2'\quad\quad\quad\quad\quad}\ar[rd]_{\mu_{A_0}\cdot T\pi_1'} && TB_1\ar[rd]_{\mu_{B_0}\cdot Td_0^B} &\\
& TA_1\ar[rr]_{T(f_0\cdot c_0^A)\quad\quad} && TB_0 \\
}$	
\end{center}

 %\begin{center}
%$\xymatrix{T(TA_1\times_{TB_0}B_1)\times_{TB_0}B_1\ar[rr]^{\quad T\pi_2'\times 1_{B_1}}\ar[dd]_{\widetilde{\pi_1}'}\ar@/^30pt/[rrrr]_{\widetilde{\pi_2}'}\ar@{..>}[rd]_{t'} && B_2=TB_1\times_{TB_0}B_1\ar[rr]^{c_1^B}\ar[dd]_{d_1^B}\ar[rd]^{m^B} && B_1\ar[dd]^{d_0^B}\\
%& TA_1\times_{TB_0}B_1\ar[rr]_{\pi_2'\quad\quad\quad\quad}\ar[dd]_{\pi_1'} && B_1\ar[dd]_{d_0^B} &\\
%T(TA_1\times_{TB_0}B_1)\ar[rr]^{T\pi_2'\quad\quad\quad\quad\quad}\ar[rd]_{\mu_{A_1}\cdot T\pi_1'} && TB_1\ar[rr]^{Tc_0^B\quad\quad\quad}\ar[rd]_{\mu_{B_0}\cdot Td_0^B} && TB_0\\
%%& TA_1\ar[rr]^{T(f_0\cdot c_0^A)\quad\quad\quad} && TB_0 & \\
%}$	
%\end{center}
One routinely shows that the left part of the following diagram commutes in the obvious sense, so that we are now able to define the morphism $c_0^P$ to make also its right part commute:
\begin{center}
$\xymatrix{D'\ar@<0.5ex>[r]^{\overline{k}}\ar@<-0.5ex>[r]_{\overline{\ell}}\ar[d]_{t'} & D\ar[r]^{\overline{z}}\ar[d]^{t} & P_1\ar@{..>}[d]^{c_0^P}\\
C'\ar@<0.5ex>[r]^k\ar@<-0.5ex>[r]_{\ell} & C\ar[r]^{ z} & P_0\\
}$	
\end{center}
Finally then, the required identity $p_0\cdot c^P_0=c_0^B\cdot p_1$ follows from
$$ p_0\cdot c_0^P\cdot\overline{z}=c_0^B\cdot \pi_2\cdot t= c_0^B\cdot m^B\cdot (T\pi_2\!\times\! 1_{B_1})=c_0^B\cdot c_1^B\cdot(T\pi_2\!\times\! 1_{B_1})=c_0^B\cdot\widetilde{\pi_2}=c_0^B\cdot p_1\cdot\overline{z}.$$

 {\em STEP 3: Defining $i^P$ and $ m^P$ and showing $c_0^P\cdot i^P=1_{P_0}$.}\\
There is only one way of defining the insertion of identities $i^P$ and the composition morphism $m^P$, so that they cooperate correctly with the domain morphisms and make $p $ a $\mathbb T$-functor:
\begin{center}
$\xymatrix{
& P_{0}\ar[dd]_{p_{0}}\ar[ld]_{\eta_{P_{0}}} &&& P_{0}\ar@{=}[lll]\ar@{.>}[ld]_{i^{P}}\ar[dd]^{p_{0}}\\
TP_{0}\ar[dd]_{Tp_{0}}&&&P_{1}\ar[lll]_{\quad\quad d^{P}_{0}}\ar[dd]^{p_{1}}&\\
&B_{0}\ar@{->}[ld]^{\eta_{B_{0}}}&&&B_{0}\ar@{=}[lll]\ar@{->}[ld]^{i^{B}}\\
TB_{0}&&&B_{1}\ar[lll]^{d^{B}_{0}}}\quad
\xymatrix{&TP_{1}\ar[dd]_{Tp_{1}}\ar[ld]_{\mu_{P_{0}}\cdot Td^{P}_{0}}&&&P_{2}\ar@{->}[lll]_{d^P_{1}}\ar@{.>}[ld]_{m^{P}}\ar[dd]^{p_{2}}\\
TP_{0}\ar[dd]_{Tp_{0}}&&&P_{1}\ar[lll]_{\quad\quad d^{P}_{0}}\ar[dd]^{p_{1}}&\\
&TB_{1}\ar@{->}[ld]^{\hs{0}\mu_{B_{0}}\cdot Td^{B}_{0}}&&&B_{2}\ar@{->}[lll]_{d^B_{1}\quad}\ar@{->}[ld]^{m^{B}}\\
TB_{0}&&&B_{1}\ar[lll]^{d^{B}_{0}}
}$
\end{center}
But to see their cooperation with $c_0^P$ requires more effort. With $C=TA_0\times_{TB_0}B_1,$ we recall that the cube on the left in the diagram below shows that its top face is a pullback, so that one has $$D=TC\times_{TB_0}B_1\cong TC\times_{TB_1}B_2,$$ and this then gives the morphism $s=\eta_C\times i_2^B$ that makes the cube on the right commute:
\begin{center}
	$\xymatrix{ & TC\ar[dd]^{T(c_0^B\cdot\pi_2)}\ar[ld]_{T\pi_2} && D\ar[ll]_{\widetilde{\pi_1}}\ar[ld]_{T\pi_2\times 1_{B_1}}\ar[dd]^{\widetilde{\pi_2}}\\
TB_1\ar[dd]_{Tc_{0}^B} && B_2\ar[ll]_{ d^B_1
\quad\quad}\ar[dd]^{c^B_1} & \\
& TB_{0}\ar@{=}[ld] && B_1\ar[ll]_{d_0^B\quad\quad\quad\quad}\ar@{=}[ld]\\
TB_{0}&& B_{1}\ar[ll]^{d^B_0} & \\
}$\quad
$\xymatrix{& C\ar[dd]_{\pi_2}\ar[ld]_{\eta_{C}} && C\ar@{=}[ll]\ar@{.>}[ld]_{s}\ar[dd]^{\pi_2}\\
TC\ar[dd]_{T\pi_2}&& D\ar[ll]_{\quad\quad\quad \widetilde{\pi_1}}\ar[dd]^{T\pi_2\times 1_{B_1}}&\\
&B_{1}\ar@{->}[ld]^{\eta_{B_{1}}}&& B_{1}\ar@{=}[ll]\ar@{->}[ld]^{i_2^{B}}\\
TB_{1}&& B_{2}\ar[ll]^{d^{B}_{1}}}$
\end{center}
Since $$d_0^P\cdot\overline{z}\cdot s=Tz\cdot\widetilde{\pi_1}\cdot s=Tz\cdot\eta_C=d_0^P\cdot i^P\cdot z,$$
$$p_1\cdot\overline{z}\cdot s=\widetilde{\pi_2}\cdot s=c_1^B\cdot(T\pi_2\times1_{B_1})\cdot s=c_1^B\cdot i_2^B\cdot\pi_2=i^B\cdot c_0^B\pi_2=i^B\cdot p_0\cdot z=p_1\cdot i^P\cdot z,$$
the upper square of the next diagram commutes, and the lower square commutes by definition of $c_0^P$:
\begin{center}
	$\xymatrix{C\ar[r]^z\ar[d]_s & P_0\ar[d]^{i^P}\\
	D\ar[r]^{\overline{z}}\ar[d]_t & P_1\ar[d]^{c_0^P}\\
	C\ar[r]^z & P_0 \\
	}$
\end{center}
Consequently, from
$$\pi_1\cdot t\cdot s=\mu_{A_0}\cdot T\pi_1\cdot\widetilde{\pi_1}\cdot s=\mu_{A_0}\cdot T\pi_1\cdot\eta_C=\mu_{A_0}\cdot\eta_{TA_0}=\pi_1,$$
$$\pi_2\cdot t\cdot s=m^B\cdot (T\pi_2\times1_{B_1})\cdot s=m^B\cdot i_2^B\cdot \pi_2=\pi_2,$$
we obtain $t\cdot s=1_C$, so that the epimorphism $z$ lets us conclude $c_0^P\cdot i^P=1_{P_0}$, as desired.

 \medskip

 {\em STEP 4: Showing $c_0^P\cdot m^P=c^P_0\cdot c_1^P$}.\\
As a pullback of the coequalizer $Tz$ of the reflexive pair $Tk,T\ell$, the morphism $\overline{z}$ is a coequalizer of the reflexive pair $\overline{k},\overline{\ell}$, and $T$ preserves it. Hence, with $$p_2=Tp_1\times p_1:P_2=TP_1\times_{TP_0}P_1\longrightarrow B_2=TB_1\times_{TB_0}B_1,$$
one pulls back the coequalizer $T\overline{z}$ of the reflexive pair $T\overline{k},T\overline{\ell}$ to obtain the (regular) epimorphism $\overline{\overline{z}}$, as shown in the commutative diagram
\begin{center}
$\xymatrix{TD\times_{TB_1}B_2\ar[r]_{\quad\quad\overline{\overline{z}}}\ar@/^15pt/[rr]^{\widehat{\pi_2}}\ar[d]_{\widehat{\pi_1}} & P_2\ar[r]_{p_2}\ar[d]^{d_1^P} & B_2\ar[d]^{d_1^B} \\
TD\ar[r]^{T\overline{z}}\ar@/_15pt/[rr]_{T\widetilde{\pi_2}} & TP_1\ar[r]^{Tp_1} & TB_1\\
}$	
\end{center}
With $\pi^T_2=T\pi_2\times1_{B_1}$, the morphism $T\pi^T_2\times 1_{B_2}$ is described by the following cube on the left, and it gives us the morphism $v$ to make the cube on the right commute:

 \begin{center}
	$\xymatrix{ & TD\ar[dd]^{T\widetilde{\pi_2}}\ar[ld]_{T\pi_2^T} && TD\!\times_{TB_1}\!B_2\ar[ll]_{\widehat{\pi_1}}\ar[ld]_{T\pi_2^T\times 1_{B_2}}\ar[dd]^{\widehat{\pi_2}}\\
TB_2\ar[dd]_{Tc_{1}^B} && B_3\ar[ll]_{ d^B_2
\quad\quad}\ar[dd]^{c^B_2} & \\
& TB_{1}\ar@{=}[ld] && B_2\ar[ll]_{d_1^B\quad\quad\quad\quad}\ar@{=}[ld]\\
TB_{1}&& B_{2}\ar[ll]^{d^B_1} & \\
}$
$\xymatrix{& TD\ar[dd]_{T\pi_2^T}\ar[ld]_{\mu_C\cdot T\widetilde{\pi_1}} && TD\!\times_{TB_1}\!B_2\ar[ll]_{\widehat{\pi_1}}\ar@{.>}[ld]_{v}\ar[dd]_{T\pi_2^T\times1_{B_2}}\\
TC\ar[dd]_{T\pi_2}&& D\ar[ll]_{\quad\quad\quad\quad \widetilde{\pi_1}}\ar[dd]^{\pi_2^T}&\\
&TB_{2}\ar@{->}[ld]^{\mu_{B_{1}}\cdot Td_1^B}&& B_{3}\ar[ll]_{d_2^B\quad\quad\quad\quad}\ar@{->}[ld]^{m_2^{B}}\\
TB_{1}&& B_{2}\ar[ll]^{d^{B}_{1}}}$
\end{center}
Parallel to $v$ we have the morphism $u$ defined by the following cube:
\begin{center}
$\xymatrix{& TD\ar[dd]_{T\pi_2^T}\ar[ld]_{Tt} && TD\!\times_{TB_1}\!B_2\ar[ll]_{\widehat{\pi_1}}\ar@{.>}[ld]_{u}\ar[dd]_{T\pi_2^T\times1_{B_2}}\\
TC\ar[dd]_{T\pi_2}&& D\ar[ll]_{\quad\quad\quad\quad \widetilde{\pi_1}}\ar[dd]^{\pi_2^T}&\\
&TB_{2}\ar@{->}[ld]^{Tm^B}&& B_{3}\ar[ll]_{d_2^B\quad\quad\quad\quad}\ar@{->}[ld]^{m_1^{B}}\\
TB_{1}&& B_{2}\ar[ll]^{d^{B}_{1}}}$	
\end{center}
The right square of the diagram
\begin{center}
$\xymatrix{TD\times_{TB_1}B_2\ar@<0.5ex>[r]^{\quad v}\ar@<-0.5ex>[r]_{\quad u}\ar[d]_{\overline{\overline{z}}} & D\ar[r]^t \ar[d]^{\overline{z}}& C\ar[d]^z\\
P_2\ar@<0.5ex>[r]^{m^P}\ar@<-0.5ex>[r]_{c_1^P} & P_1\ar[r]^{c_0^P} & P_0\\
}$	
\end{center}
commutes by definition of $c_0^P$, and one routinely confirms the identities $m^P\cdot \overline{\overline{z}}=\overline {z}\cdot v$ and $c_1^P\cdot \overline{\overline{z}}=\overline{z}\cdot u$, by post-composing both sides with the pullback projections $d_0^P$ and $ p_1$ of $P_1$. Since $\overline{\overline{z}}$ is an epimorphism, it now suffices to show $t\cdot v=t\cdot u$, in order for us to conclude the desired identity $c_0^P\cdot m^P=c_0^P\cdot c_1^P$. For that, one invokes the associativity of both, the monad multiplication $\mu$ and the composition morphism $m^B$, as follows:
\begin{align*}\pi_1\cdot t\cdot u &= \mu_{A_0}\cdot T\pi_1\cdot \widetilde{\pi_1}\cdot u & \pi_2\cdot t\cdot u &= m^B\cdot \pi_2^T\cdot u\\
& =\mu_{A_0}\cdot T\pi_1\cdot Tt\cdot\widehat{\pi_1} && = m^B\cdot m_1^B\cdot (T\pi_2^T\times 1_{B_2}) \\
& = \mu_{A_0}\cdot T\mu_{A_0}\cdot TT\pi_1\cdot T\widetilde{\pi_1}\cdot\widehat{\pi_1} && = m^B\cdot m_2^B\cdot (T\pi_2^T\times 1_{B_2})\\
& = \mu_{A_0}\cdot \mu_{TA_0}\cdot TT\pi_1\cdot T\widetilde{\pi_1}\cdot\widehat{\pi_1} && =m^B\cdot\pi_2^T\cdot v\\
&=\mu_{A_0}\cdot T\pi_1\cdot \mu_C\cdot T\widetilde{\pi_1}\cdot\widehat{\pi_1} && =\pi_2\cdot t\cdot v\,.\\
&=\mu_{A_0}\cdot T\pi_1\cdot\widetilde{\pi_1}\cdot v\\
&=\pi_1\cdot t\cdot v\,,\\
\end{align*}

 {\em STEP 5: Finishing the proof that $P$ is a $\mathbb T$-category and $p:P\longrightarrow B$ a $\mathbb T$-functor.}\\ For that proof, only the verification of the three identities
$$m^P\cdot i_1^P=1_{P_0}=m^P\cdot i_2^P\;\text{ and }m^P\cdot m_1^P=m^P\cdot m_2^P$$
is still outstanding. Using the identities (10) for $p$, we confirm the last of these three identities, by post-composing both sides with the pullback projections of $P_1$:
\begin{align*}
d_0^P\cdot m^P\cdot m_2^P & = \mu_{P_0}\cdot Td_0^P\cdot d_1^P\cdot m_2^P & p_1\cdot m^P\cdot m_2^P & = m^B\cdot p_2\cdot m_2^P\\
& =\mu_{P_0}\cdot Td_0^P\cdot\mu_{P_1}\cdot Td_1^P\cdot d_2^P && =m^B\cdot m_2^B\cdot p_3\\
& =\mu_{P_0}\cdot\mu_{TP_0}\cdot TTd_0^P\cdot Td_1^P\cdot d_2^P && =m^B\cdot m_1^B\cdot p_3\\
& =\mu_{P_0}\cdot T(\mu_{P_0}\cdot Td_0^P\cdot d_1^P)\cdot d_2^P && =m^B\cdot p_2\cdot m_1^P\\\
& =\mu_{P_0}\cdot T(d_0^P\cdot m^P)\cdot d_2^P && =p_1\cdot m^P\cdot m_1^P\;.\\
& =\mu_{P_0}\cdot Td_0^P\cdot d_1^P\cdot m_1^P &&\\
& =d_0^P\cdot m^P\cdot m_1^P\;,&&\\
\end{align*}
The confirmation of the first two of the three claimed identities proceeds very similarly.

 \medskip

 {\em STEP 6: Establishing the factorization $f=r\cdot p$.}\\ With $e:=\eta_{A_0}\times i^B$ as in the cube below on the left, one defines
$$r_0:= z\cdot e\quad\text{and}\quad r_1:= Tr_0\times f_1,$$
where the defining commutativity conditions for $r_1$ are shown by the cube on the right:

 \begin{center}
$\xymatrix{& A_0\ar[dd]_{f_0}\ar[ld]_{\eta_{A_0}} && A_0\ar@{=}[ll]\ar@{.>}[ld]_{e}\ar[dd]^{f_0}\\
TA_0\ar[dd]_{Tf_0}&& C\ar[ll]_{\quad\quad\quad {\pi_1}}\ar[dd]^{\pi_2}&\\
&B_{0}\ar@{->}[ld]^{\eta_{B_{0}}}&& B_{0}\ar@{=}[ll]\ar@{->}[ld]^{i^{B}}\\
TB_{0}&& B_{1}\ar[ll]^{d^{B}_{0}}}$
\quad
$\xymatrix{& TA_0\ar@{=}[dd]\ar[ld]_{Tr_0} && A_1\ar[ll]_{d_0^A}\ar[ld]_{r_1}\ar@{=}[dd]_{}\\
TP_0\ar[dd]_{Tp_0}&& P_1\ar[ll]_{\qquad\quad d_0^P}\ar[dd]^{p_1}&\\
&TA_{0}\ar@{->}[ld]^{Tf_0}&& A_{1}\ar[ll]_{d_0^A\quad\quad}\ar@{->}[ld]^{f_1}\\
TB_{0}&& B_{1}\ar[ll]^{d^{B}_{0}}}$	
\end{center}
By definition, one has $p_1\cdot r_1=f_1$, and $p_0\cdot r_0=p_0\cdot z\cdot e=c_0^B\cdot\pi_2\cdot e=c_0^B\cdot i^B\cdot f_0= f_0$, as well as
$d_0^P\cdot r_1=Tr_0\cdot d_0^A$. In order to show $c_0^P\cdot r_1=r_0\cdot c_0^A$, we first factor
$$r_1=Tr_0\times f_1=(Tr_0\times 1_{B_1})\cdot (1_{TA_0}\times f_1)=(Tz\times 1_{B_1})\cdot(Te\times 1_{B_1})\cdot(1_{TA_0}\times f_1),$$
where $Tz\times 1_{B_1}=\overline{z}$, and where $1_{TA_0}\times f_1: A_1\longrightarrow C,\; Te\times 1_{B_1}:C\longrightarrow D$ satisfy
$$\pi_1\cdot(1_{TA_0}\times f_1)=d_0^A,\quad\pi_2\cdot(1_{TA_0}\times f_1)=f_1,$$
$$\widetilde{\pi_1}\cdot (Te\times 1_{B_1})=Te\cdot\pi_1,\quad(T\pi_2\times 1_{B_1})\cdot(Te\times 1_{B_1})=i_1^B\cdot\pi_2.$$
With the defining condition for $c_0^P$ of Step 2 we now obtain
$$c_0^P\cdot r_1= c_0^P\cdot\overline{z}\cdot(Te\times 1_{B_1})\cdot(1_{TA_0}\times f_1)=z\cdot t\cdot(Te\times 1_{B_1})\cdot(1_{TA_0}\times f_1).$$
But by post-composing $t\cdot (Te\times 1_{B_1}): C\longrightarrow C$ with the pullback projections $\pi_1,\pi_2$, one sees that this morphism must equal $1_C$, so that we have
$c_0^P\cdot r_1=z\cdot (1_{TA_0}\times f_1) $. Hence, the required compatibility of $(r_0,r_1)$ with the codomain morphisms will follow once we have shown $z\cdot (1_{TA_0}\times f_1)=r_0\cdot c_0^A$; that is: once we have confirmed the commutativity of the right square of the diagram
\begin{center}
$\xymatrix{A_2\ar@<0.5ex>[r]^{m^A}\ar@<-0.5ex>[r]_{c_1^A}\ar[d]_{1_{TA_1}\times f_1} & A_1\ar[r]^{c_0^A}\ar[d]^{1_{TA_0}\times f_1} & A_0\ar[d]^{r_0}\\
C'\ar@<0.5ex>[r]^k\ar@<-0.5ex>[r]_{\ell} & C\ar[r]^{ z} & P_0\\
}$	
\end{center}

 To this end one first establishes quite routinely that its left part commutes in the obvious sense. This then gives
$z\cdot (1_{TA_0}\times f_1)\cdot m^A= z\cdot(1_{TA_0}\times f_1)\cdot c_1^A$. In order to conclude the proof of the commutativity of the right square, we employ the contractibility of the coequalizer at the top of the diagram and must therefore just show the equality $z\cdot (1_{TA_0}\times f_1)\cdot i^A=r_0$.
But, as one easily confirms, one has $(1_{TA_0}\times f_1)\cdot i^A= e$, and the equality $z\cdot e=r_0$ holds by definition of $r_0$.

 \medskip

 {\em STEP 7: Confirming the $\mathbb T$-functoriality of $r$.}\\ In order to show $r_1\cdot i^A=i^P\cdot r_0$ and $r_1\cdot m^A=m^P\cdot r_2$, we take advantage of the fact that $p$ is perfect, so that it suffices to validate
these equalities after being post-composed by the pullback projections of $P_1=TP_0\times_{TB_0} B_1$, as follows:
\begin{align*}
d_0^P\cdot r_1\cdot i^A &=	Tr_0\cdot d_0^A\cdot i^A & p_1\cdot r_1\cdot i^A &=f_1\cdot i^A\\
&=Tr_0\cdot\eta_{A_0} & & =i^B\cdot f_0\\
&=\eta_{P_0}\cdot r_0 & &=i^B\cdot p_0\cdot r_0\\
&=d_0^P\cdot i^P\cdot r_0 &&=p_1\cdot i^P\cdot r_0\\
&&&\\
d_0^P\cdot r_1\cdot m^A &= Tr_0\cdot d_0^A\cdot m^A & p_1\cdot r_1\cdot m^A &=f_1\cdot m^A\\
&=Tr_0\cdot\mu_{A_0}\cdot Td_0^A\cdot d_1^A &&=m^B\cdot f_2\\
&=\mu_{P_0}\cdot Td_0^P\cdot Tr_1\cdot d_1^A &&=m^B\cdot p_2\cdot r_2\\
&=\mu_{P_0}\cdot Td_0^P\cdot d_1^P\cdot r_2 &&=p_1\cdot m^P\cdot r_2\\
&=d_0^P\cdot m^P\cdot d_1^P\cdot r_2 &&\\
\end{align*}
%consider the diagram
%\begin{center}
%	$\xymatrix{A_0\ar[r]^e\ar[d]_{i^A} & C\ar[r]^z\ar[d]^s & P_0\ar[d]^{i^P}\\
%	A_1\ar[r]^{Te\times f_1} & D\ar[r]^{\overline{z}} &P_1
%	}$
%\end{center}
%whose horizontal composite arrows equal $r_0$ (top) and $r_1$ (bottom). The commutativity of its right square was established in Step 3. Since $Te\times f_1=(Te\times 1_{B_1})\cdot (_{TA_0}\times f_1)$, the following computation confirms that its left square commutes as well, which then gives the required equality:
%\begin{align*}
%%\widetilde{\pi_1}\cdot (Te\times 1_{B_1})\cdot (1_{TA_0}\times f_1)\cdot i^A &=Te\cdot \pi_1\cdot (1_{TA_0}\times f_1)\cdot i^A\\ &=Te\cdot d_0^A\cdot i^A\\
%&=Te\cdot \eta_{A_0}\\
%%&=\eta_C\cdot e\\
%&= \widetilde{\pi_1}\cdot s\cdot e\;;
%&\\
%(T\pi_2\times 1_{B_1})\cdot (Te\times 1_{B_1})\cdot (1_{TA_0}\times f_1)\cdot i^A &=i_1^B\cdot\pi_2\cdot (1_{TA_0}\times f_1)\cdot i^A\\
%&=i_1^B\cdot f_1\cdot i^A\\
%&=i_1^B\cdot i^B\cdot f_0\\
%&=i_2^B\cdot i^B\cdot f_0\\
%&=i_2^B\cdot \pi_2\cdot e\\
%&= (T\pi_2\times 1_{B_1})\cdot s\cdot e\;.\\
%\%end{align*}
%For the verification of $r_1\cdot m^A=m^P\cdot r_2$ we proceed similarly, by establishing the commutativity of the diagram
%\begin{center}
%$\xymatrix{A_2\ar[r]^{y\qquad}\ar[d]_{m^A} & TD\times_{TB_1}B_2\ar[r]^{\qquad \overline{\overline{z}}}\ar[d]^v & P_2\ar[d]^{m^P}\\
%A_1\ar[r]^{Te\times f_1} & D \ar[r]^{\overline{z}}& P_1\\
%}$	
%\end{center}
%with its upper horizontal composite arrow equaling $r_2$. Here $y$ is determined by ....

 \medskip
{\em STEP 8: Proving the universality---the existence part.}\\ For any perfect $\mathbb T$-functor $q:Q\longrightarrow B$, we need to show that morphisms $g:f\longrightarrow q$ in ${\sf Cat}(\mathbb T)/B$ correspond bijectively to morphisms $h:p\longrightarrow q$ in ${\sf Cat}(\mathbb T)/B$, via $g=h\cdot r$. To this end, given $g$, we obtain $h_0$ with the help of the diagram
\begin{center}
$\xymatrix{C'\ar@<0.5ex>[r]^k\ar@<-0.5ex>[r]_{\ell}\ar[d]_{\check{\check{g}}} & C\ar[r]^z\ar[d]^{\check{g}} & P_0\ar@{..>}[d]^{h_0}\\
Q_2\ar@<0.5ex>[r]^{m^Q}\ar@<-0.5ex>[r]_{c_1^Q} & Q_1\ar[r]_{c_0^Q} & Q_0\\
}$	
\end{center}
where $\check{g}=Tg_0\times 1_{B_1}$ and $\check{\check{g}}=Tg_1\times 1_{B_1}$; these morphisms exist since the perfect $\mathbb T$-functor $q$ makes the front faces of both diagrams below pullbacks:
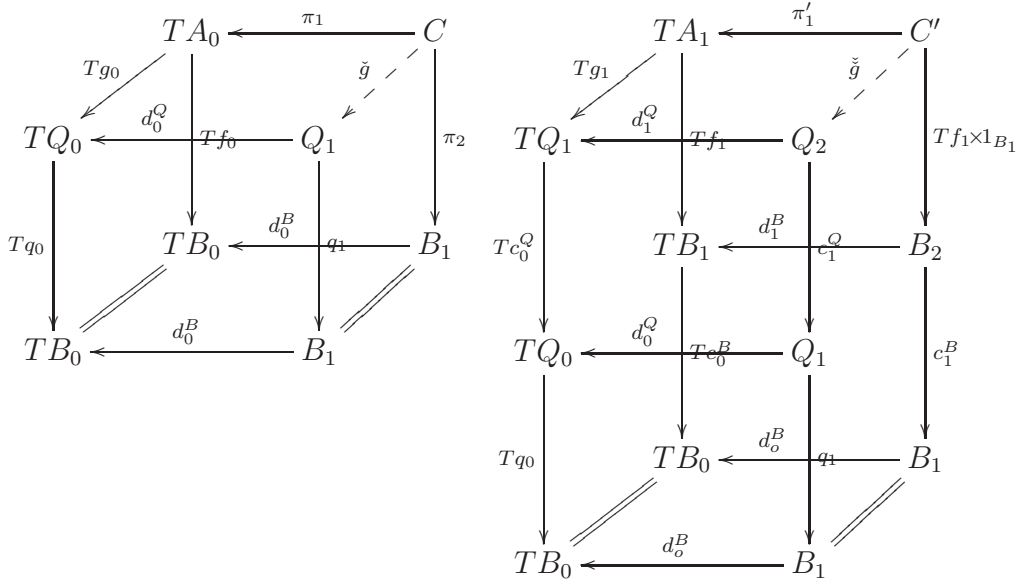
\begin{center}
$\xymatrix{ & TA_0\ar[dd]^{Tf_0}\ar[ld]_{Tg_0} && C\ar[ll]_{{\pi_1}}\ar@{-->}[ld]_{\check{g}}\ar[dd]^{{\pi_2}}\\
TQ_0\ar[dd]_{Tq_0} && Q_1\ar[ll]_{ d_0^Q
\quad\quad}\ar[dd]^{q_1} & \\
& TB_{0}\ar@{=}[ld] && B_1\ar[ll]_{d_0^B\quad\quad}\ar@{=}[ld]\\
TB_{0}&& B_{1}\ar[ll]_{d^B_0} & \\
}$
$\xymatrix{ & TA_1\ar[ld]_{Tg_1}\ar[dd]^{Tf_1} && C'\ar[ll]_{\pi_1'}\ar@{-->}[ld]_{\check{\check{g}}}\ar[dd]^{Tf_1\!\times\!1_{B_1}}\\
TQ_1\ar[dd]_{Tc_0^Q} && Q_2\ar[ll]_{d_1^Q\qquad}\ar[dd]^{c_1^Q} &\\
& TB_1\ar[dd]^{Tc_0^B} && B_2\ar[ll]_{d_1^B\qquad}\ar[dd]^{c_1^B}\\
TQ_0\ar[dd]_{Tq_0} && Q_1
\ar[ll]_{d_0^Q\qquad}\ar[dd]^{q_1} &\\
& TB_0\ar@{=}[ld] && B_1\ar[ll]_{d_o^B\qquad}\ar@{=}[ld]\\
TB_0 && B_1\ar[ll]_{d_o^B} &\\
}$	
\end{center}
Checking the equalities $\check{g}\cdot k= m^Q\cdot \check{\check{g}}$ and $\check{g}\cdot\ell=c_1^Q\cdot\check{\check{g}}$ is a routine matter, and they then secure the existence of the morphism $h_0$ with $h_0\cdot z= c_0^Q\cdot\check{g}$. Since
$$q_0\cdot h_0\cdot z=q_0\cdot c_0^Q\cdot\check{g}=c_0^B\cdot q_1\cdot\check{g}=c_0^B\cdot\pi_2=p_0\cdot z,$$
$q_0\cdot h_0=p_0$ follows, so that $h_0:p_0\longrightarrow q_0$ is a morphism in $\C/B_0$. Moreover, with the easily established equality
$$\check{g}\cdot e= i^Q\cdot g_0$$
one sees that $g_0$ factors as desired: $h_0\cdot r_0=h_o\cdot z\cdot e=c_0^Q\cdot\check{g}\cdot e=c_0^Q\cdot i^Q\cdot g_0=g_0$.

 As $h_0$ should be part of a morphism $h:p\longrightarrow q$ in ${\sf Cat}(\mathbb T)/B$, we are forced to set $$h_1 = Th_0\times 1_{B_1},$$
as described by the following cube on the left:
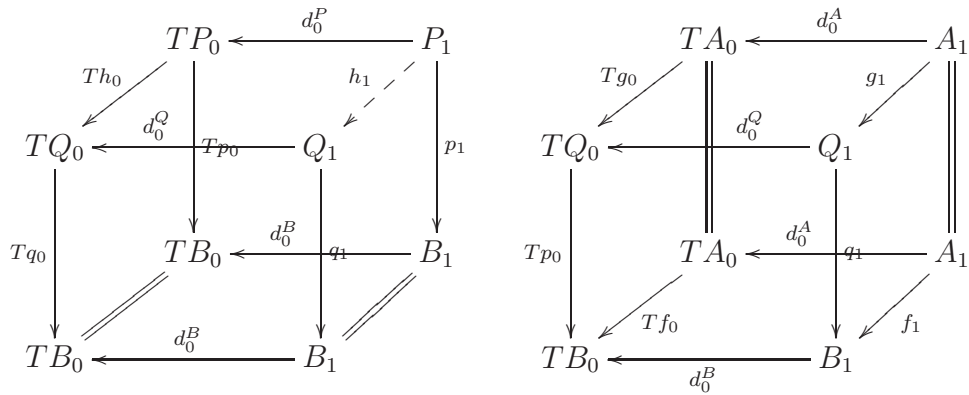
\begin{center}
$\xymatrix{ & TP_0\ar[dd]^{Tp_0}\ar[ld]_{Th_0} && P_1\ar[ll]_{{d_0^P}}\ar@{-->}[ld]_{h_1}\ar[dd]^{p_1}\\
TQ_0\ar[dd]_{Tq_0} && Q_1\ar[ll]_{ d_0^Q
\quad\quad}\ar[dd]^{q_1} & \\
& TB_{0}\ar@{=}[ld] && B_1\ar[ll]_{d_0^B\quad\quad}\ar@{=}[ld]\\
TB_{0}&& B_{1}\ar[ll]_{d^B_0} & \\
}$
\quad
$\xymatrix{& TA_0\ar@{=}[dd]\ar[ld]_{Tg_0} && A_1\ar[ll]_{d_0^A}\ar[ld]_{g_1}\ar@{=}[dd]_{}\\
TQ_0\ar[dd]_{Tp_0}&& Q_1\ar[ll]_{\qquad\quad d_0^Q }\ar[dd]^{q_1}&\\
&TA_{0}\ar@{->}[ld]^{Tf_0}&& A_{1}\ar[ll]_{d_0^A\quad\quad}\ar@{->}[ld]^{f_1}\\
TB_{0}&& B_{1}\ar[ll]^{d^{B}_{0}}}$	

 \end{center}
By definition, one then has $h_1:p_1\longrightarrow q_1$ in $\C/B_1$ and, with the above cube on the right, $$h_1\cdot r_1=(Th_0\times 1_{b_1})\cdot(Tr_0\times f_1)=Tg_0\times f_1=g_1$$
follows.

 By definition, $h_0,h_1$ are compatible with the domain morphisms. The corresponding statement for the codomain morphisms may be shown as follows. First one defines $\hat{g}=T\check{g}\times 1_{B_2}$ to render the cube on the left commutative and then, using the discrete opfibration $q$, one shows that the two squares on the right commute:

 \begin{center}
$\xymatrix{& TC\ar[ld]_{T\check{g}}\ar[dd]^{T\pi_2} && D\ar[ll]_{\widetilde{\pi_1}}\ar@{-->}[ld]_{\hat{g}}\ar[dd]^{T\pi_2\times1_{B_1}=\pi_2^T} \\
TQ_1\ar[dd]_{Tq_1} && Q_2\ar[ll]_{d_1^Q\qquad}\ar[dd]^{q_2} &\\
& TB_1\ar@{=}[ld] && B_2\ar[ll]_{d_1^B\qquad}\ar@{=}[ld]\\
TB_1 && B_2\ar[ll]_{d_1^B}\\
}	
\hfil
\xymatrix{D\ar[d]_{\hat{g}}\ar[r]^{\overline{z}}& P_1\ar[d]^{h_1}\\
Q_2\ar[r]^{c_1^Q} & Q_1\\
D\ar[d]_{\hat{g}}\ar[r]^{t} & C\ar[d]^{\check{g}}\\
Q_2\ar[r]^{m^Q} & Q_1\\
}$
\end{center}
Now we have
$$c_0^Q\cdot h_1\cdot\overline{z}=c_0^Q\cdot c_1^Q\cdot\hat{g}=c_0^Q\cdot m^Q\cdot\hat{g}=c_0^Q\cdot \check{g}\cdot t=h_0\cdot z\cdot t=h_0\cdot c_0^P\cdot\overline{z}\;,$$
which implies the desired equality $c_0^Q\cdot h_1=h_0\cdot c_0^P$, since $\overline{z}$ is epic.

 Checking the $\mathbb T$-functoriality conditions for $h$ is a routine computation which, again, just uses the pullback $Q_1=TQ_0\times_{TB_0}B_1$. This completes the proof that every
$g:f\longrightarrow q$ in ${\sf Cat}(\mathbb T)/B$ gives a morphism $h:p\longrightarrow q$ in ${\sf Cat}(\mathbb T)/B$ with $g=h\cdot r$. It remains to be shown that such $h$ depends uniquely on $g$.
\medskip

 {\em STEP 9: Proving the universality---the uniqueness part.}\\
For any morphism $\tilde{h}:p\longrightarrow q$ in ${\sf Cat}(\mathbb T)/B$ with $g=\tilde{h}\cdot r$ one has $\tilde{h}_1=T\tilde{h}_0\times 1_{B_1}$. Therefore, it suffices to show $\tilde{h}_0=h_0$. To this end we first factor $\check{g}$ as
$$\tilde{h}_1\cdot (Tr_0\times1_{B_1})=(T\tilde{h}_0\times 1_{B_1})\cdot (Tr_0\times 1_{B_1})=Tg_0\times 1_{B_1}=\check{g}$$
and then, using $t\cdot (Te\times1_{B_1})=1_C$ as shown in Step 6, conclude
\begin{align*}
\tilde{h}_0\cdot z &=\tilde{h}_0\cdot z\cdot t\cdot (Te\times1_{B_1})\\
&=\tilde{h}_0\cdot c_0^P\cdot\overline{z}\cdot(Te\times1_{B_1})\\
&=c_0^Q\cdot\tilde{h}_1\cdot(Tz\times1_{B_1})\cdot(Te\times1_{B_1})\\
&=c_0^Q\cdot\tilde{h}_1\cdot(Tr_0\times1_{B_1})\\
&=c_0^Q\cdot\check{g}\\
&=h_0\cdot z\;.\\
\end{align*}
Indeed, since $z$ is epic, $\tilde{h}_0=h_0$ follows.
\end{proof}

\section{The comprehensive factorization system for the category of T-categories}

 We first state a general fact on (orthogonal) factorization systems (as defined in \cite{FreydKelly1972} and, less redundantly, in \cite{AHS}); for the elegant presentation of their strict cousins in terms of distributive laws, see \cite{RosebrughWood2002}. Variations of the following proposition appeared in various forms early on, in \cite{EhrbarWyler1968}, and then in \cite{Tholen1979} (Lemma 7.3), \cite{MacDonaldTholen1982} (Proposition 1.2), and \cite{DikranjanTholen1995} (Theorem 1.8); here we formulate and prove a concise and self-contained version of it, as follows:

 \begin{proposition}\label{factsystems} A class $\M$ of morphisms in a category $\A$ with pullbacks belongs to an orthogonal factorization system $(\E,\M)$ in $\A$ if, and only if,

 {\em 1.} $\M$ is closed under composition;

 {\em 2.} $\M$ is stable under pullback;

 {\em 3.} the full subcategory $\M/B$ of the slice category $\A/B$ is reflective, for every $B$ in $\A$. 	
\end{proposition}

 \begin{proof}
The necessity of the three conditions for $\M$ to be part of an orthogonal factorization system is well known. For their sufficiency, let us first note that its pullback stability makes $\M$ closed under pre-composition with isomorphisms. Given a morphism $f:A\longrightarrow B$, considered as an object of $\A/B$, its reflection into $\M/B$ gives us a factorization $f=e\cdot m$ with $m\in\M$, and it now suffices to show that $e$ belongs to the left orthogonal complement of $\M$ in $\A$, here called $\E$, so that every solid-arrow commutative square
\begin{center}
$\xymatrix{A\ar[r]^u\ar[d]_e & D\ar[d]^n\\
C\ar[r]_v\ar@{-->}[ru]^w & E\\
}$	
\end{center}
with $n \in \M$ admits a unique fill-in arrow $w$. With the pullback $(n':P\rightarrow C,\, v':P\rightarrow D)$ of $(v,n)$, we have $m\cdot n'\in \M$ by conditions 1 and 2. The unique morphism $k:A\longrightarrow P$ with $n'\cdot k=e$ and $v'\cdot k= u$ constitutes a morphism $f\longrightarrow m\cdot n'$ in $\A/B$, so that the reflection property of condition 3 gives us a unique morphism $t: C\longrightarrow P$ with $t\cdot e= k$ and $m\cdot n'\cdot t=m$. In conjunction with $n'\cdot t\cdot e= n'\cdot k =e$ one concludes $n'\cdot t= 1_C$, which then shows that $w:=v'\cdot t$ makes the above diagram commute.

 For any other morphism $z:C\longrightarrow D$ with $z\cdot e=e$ and $n\cdot z=v$ one has the morphism $s:C\longrightarrow P$ with $n'\cdot s=1_C$ and $v'\cdot s= z$, and we see that $s\cdot e$ satisfies the defining conditions of $k$. So, since $s\cdot e = k$ and $m\cdot n'\cdot s=m$, we have $s=t$, which then shows $z=v'\cdot s=v'\cdot t= w$.
\end{proof}

 In order to employ Proposition \ref{factsystems} when $\A={\sf Cat}(\mathbb T)$ and $\M={\sf Perf}(\mathbb T)$, we note that $\M$ is trivially closed under composition, and that Theorem \ref{reflective} proves the reflectivity condition, provided that $\C$ has coequalizers of reflexive pairs that are stable under pullback and that are preserved by $T$. Hence, we are left with the task of having to show the existence of pullbacks in $\A$ and the pullback-stability of $\M$ --- a property that follows readily from Corollary \ref{Tpreservespbs} when $T$ preserves pullbacks, but that is much harder to establish when {\em we do {\em not} assume that $T$ preserves pullbacks}, as follows.

\begin{proposition}\label{pbstability} The category ${\sf Cat}(\mathbb T)$ has pullbacks, and the class of perfect $\mathbb T$-functors is stable under them.	
\end{proposition}

 \begin{proof}
For $\mathbb T$-functors $f:A\longrightarrow C$ and $g:B\longrightarrow C$, initially following the construction given in Corollary \ref{Tpreservespbs}, we build the diagram below:

 \begin{center}
$\xymatrix{TP_0\ar@/_30pt/[dd]_{Tg_0'}\ar[d]_{\ell_0}\ar[rrddd]^{Tf_0'} &&& P_1\ar[lll]_{d_0^P}\ar@/_30pt/[dd]_{g_1'}\ar[d]_{\ell_0'}\ar[rrddd]^{f_1'}\ar[rrr]^{c_0^P} &&& P_0\ar@/_30pt/[dd]_{g_0'}\ar@{=}[d]\ar[rrddd]^{f_0'} &&\\
P_0^T\ar[d]_{g_0^T}\ar[rrddd]^{f_0^T} &&& \widetilde{P_1}\ar[lll]_{\widetilde{d_0}}\ar[d]_{g_1^*}\ar[rrddd]^{f_1^*}\ar[rrr]^{\widetilde{c_0}} &&& P_0 \ar[d]_{g_0'}\ar[rrddd]^{f_0'} &&\\
TA_0\ar[rrddd]_{Tf_0} &&& A_1\ar[lll]_{d_0^A}\ar[rrddd]_{f_1}\ar[rrr]^{c_0^A} &&& A_0\ar[rrddd]_{f_0} &&\\
&& TB_0\ar@{=}[d] &&& B_1\ar[lll]_{d_0^B\qquad}\ar@{=}[d]\ar[rrr]^{c_0^B\qquad} &&& B_0\ar@{=}[d]\\
&& TB_0\ar[d]^{Tg_0} &&& B_1\ar[lll]_{d_0^B\qquad}\ar[d]^{g_1}\ar[rrr]^{c_0^B\qquad} &&& B_0\ar[d]^{g_0}\\
&& TC_0 &&& C_1\ar[lll]_{d_0^C\qquad}\ar[rrr]^{c_0^C\qquad} &&& C_0\\
}$
\end{center}
As in Corollary	\ref{Tpreservespbs}, the vertical right, central and left panels of its lower storey are formed by pullbacks, so that $P_0:=A_0\times_{C_0}B_0,\; \widetilde{P_1}:=A_1\times_{C_1}B_1$ and $P_0^T:=TA_0\times_{TC_0}TB_0$, with induced morphisms $\widetilde{d_0}$ and $\widetilde{c_0}$. The only additional step to be taken now is that the comparison morphism
$$\ell_0 :TP_0=T(A_0\times_{C_0}B_0)\longrightarrow P_0^T=TA_0\times_{TC_0}TB_0$$
(which was assumed to be an isomorphism in Corollary \ref{Tpreservespbs}) needs to be pulled back along $\widetilde{d_0}$ and thereby defines $P_1:=TP_0\times_{P_0^T}\widetilde{P_1}$. The pullback comes with projections $d_0^P$ and $\ell_0'$ and gives the composite morphisms
$$c_0^P:=\widetilde{c_0}\cdot\ell_0',\quad f_1':=f_1^*\cdot\ell_0',\quad g_1':=g_1^*\cdot\ell_0'.$$
As outlined below, we will show in nine steps that the $\mathbb T$-graph $(P_0,P_1,d_0^P,c_0^P)$ carries a $\mathbb T$-category structure that makes
$f'=(f_0', f_1'):P\longrightarrow A$ and $g'=(g_0', g_1'):P\longrightarrow B$ $\mathbb T$-functors and, in fact, gives the desired pullback in ${\sf Cat}(\mathbb T)$. Then the pullback stability of ${\sf Perf}(\mathbb T)$ follows easily. Indeed, if $g$ is perfect, so that the lower left front panel of the diagram above is a pullback, also the lower left back panel is a pullback since the sides are pullbacks by design. But also the upper left back panel is a pullback, making the entire left back panel a pullback and, hence, showing that $g'$ is perfect.

 \medskip

 {\em STEP 1: Defining $i^P:P_0\longrightarrow P_1$.}\\
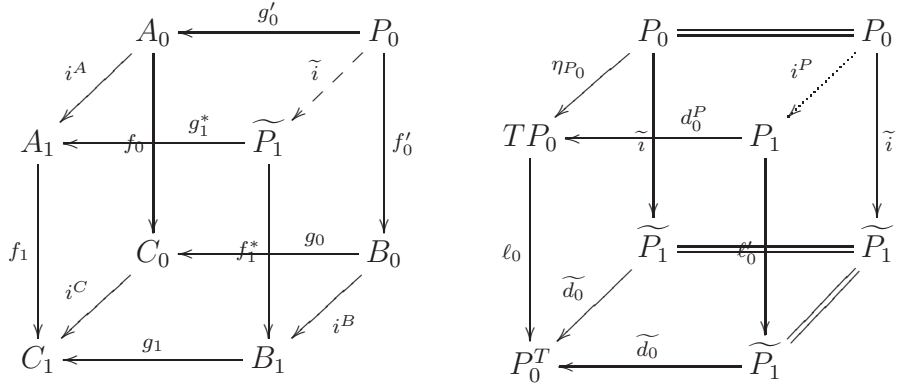
\begin{center}
$\xymatrix{ & A_0\ar[ld]_{i^A}\ar[dd]_{f_0} && P_0\ar[ll]_{g_0'}\ar@{-->}[ld]_{\widetilde{i}}\ar[dd]^{f_0'}\\
A_1\ar[dd]_{f_1} && \widetilde{P_1}\ar[ll]_{\qquad\quad g_1^*}\ar[dd]_{f_1^*}\\
& C_0\ar[ld]_{i^C} && B_0\ar[ll]_{\qquad\quad g_0}\ar[ld]^{i^B}\\
C_1 && B_1\ar[ll]_{g_1} &\\
}$
\qquad
$\xymatrix{& P_0\ar[ld]_{\eta_{P_0}}\ar[dd]_{\widetilde{i}} && P_0\ar@{=}[ll]\ar@{..>}[ld]_{i^P}\ar[dd]^{\widetilde{i}}\\
TP_0\ar[dd]_{\ell_0} && P_1\ar[ll]_{\qquad\quad d_0^P}\ar[dd]_{\ell_0'}\\
& \widetilde{P_1}\ar[ld]_{\widetilde{d_0}}&& \widetilde{P_1}\ar@{=}[ll]\ar@{=}[ld]\\
P_0^T && \widetilde{P_1}\ar[ll]_{\widetilde{d_0}}\\
}$	
\end{center}
For $\widetilde{i}:=i^A\times i^B$ as described by the cube on the left one easily confirms the equality $\widetilde{d_0}\cdot\widetilde{i}=\ell_0\cdot \eta_{P_0}$, so that one can put $i^P:= \eta_{P_0}\times 1_{\widetilde{P_1}}$ as described by the cube on the right. By definition, one has $d_0^P\cdot i^P=\eta_{P_0}$. Furthermore, showing first $\widetilde{c_0}\cdot\widetilde{i}=1_{P_0}$ we conclude $c_0^P\cdot i^P=\widetilde{c_0}\cdot\ell_0'\cdot i^P=\widetilde{c_0}\cdot\widetilde{i}=1_{P_0}$, as required.

 \medskip

 {\em STEP 2: Establishing the pullback $\widetilde{P_2}$ of $(f_2, g_2)$ over $C_2$ as a pullback over $P_0^T$.}
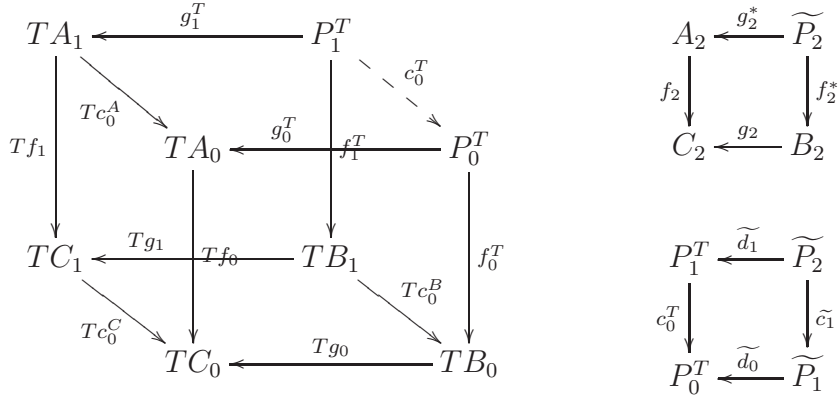
\begin{center}
$\xymatrix{TA_1\ar[dd]_{Tf_1}\ar[rd]_{Tc_0^A} && P_1^T\ar[ll]_{g_1^T}\ar[dd]^{f_1^T}\ar@{-->}[rd]^{c_0^T} &\\
& TA_0\ar[dd]^{Tf_0} && P_0^T\ar[ll]_{g_0^T\qquad\quad}\ar[dd]^{f_0^T}\\
TC_1\ar[rd]_{Tc_0^C} && TB_1\ar[ll]_{Tg_1\qquad\quad}\ar[rd]^{Tc_0^B} &\\
& TC_0 && TB_0\ar[ll]_{Tg_0}\\
}$
\qquad\qquad
$\xymatrix{A_2\ar[d]_{f_2} & \widetilde{P_2}\ar[l]_{g_2^*}\ar[d]^{f_2^*}\\
C_2 & B_2\ar[l]_{g_2}\\
P_1^T\ar[d]_{c_0^T} & \widetilde{P_2}\ar[l]_{\widetilde{d_1}}\ar[d]^{\widetilde{c_1}} \\
P_0^T & \widetilde{P_1}\ar[l]_{\widetilde{d_0}}\\
}$	
\end{center}
Like the front panel, also the back panel of the cube is a pullback by definition and, thus, determines the (suggestively, but slightly illegitimately, named) morphism $c_0^T:=Tc_0^A\times Tc_0^B$. The upper square on the right is, by definition, a pullback as well. The next goal is to establish the fact that the lower square is also a pullback (with the same pullback object $\widetilde{P_2}$), where $\widetilde{d_1},\widetilde{c_1}$ are defined to make the following diagram commutative:
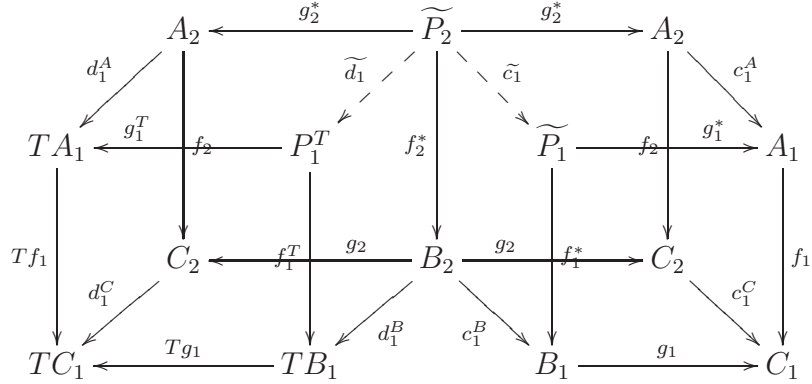
\begin{center}
$\xymatrix{& A_2\ar[ld]_{d_1^A}\ar[dd]^{f_2} && \widetilde{P_2}\ar[ll]_{g_2^*}\ar@{-->}[ld]_{\widetilde{d_1}}\ar[dd]_{f_2^*}\ar@{-->}[rd]^{\widetilde{c_1}}\ar[rr]^{g_2^*} && A_2\ar[dd]_{f_2}\ar[rd]^{c_1^A} &\\
TA_1\ar[dd]_{Tf_1} && P_1^T\ar[ll]_{g_1^T\qquad\quad}\ar[dd]_{f_1^T} && \widetilde{P_1}\ar[dd]^{f_1^*}\ar[rr]^{\qquad\quad g_1^*} && A_1\ar[dd]^{f_1}\\
& C_2\ar[ld]_{d_1^C} && B_2\ar[ll]_{\qquad\quad g_2}\ar[ld]^{d_1^B}\ar[rd]_{c_1^B}\ar[rr]^{g_2\qquad\quad} && C_2\ar[rd]^{c_1^C} &\\
TC_1 && TB_1\ar[ll]_{Tg_1} && B_1\ar[rr]^{g_1} && C_1\\
}$
\end{center}
The claim that the upper small diamond of the following diagram commutes and is in fact a pullback now follows formally from the given pullback presentations of each of its four objects (drawn by dotted arrows) and the fact that all other parts of the diagram commute, with the other three small diamonds of the diagram being pullbacks as well.
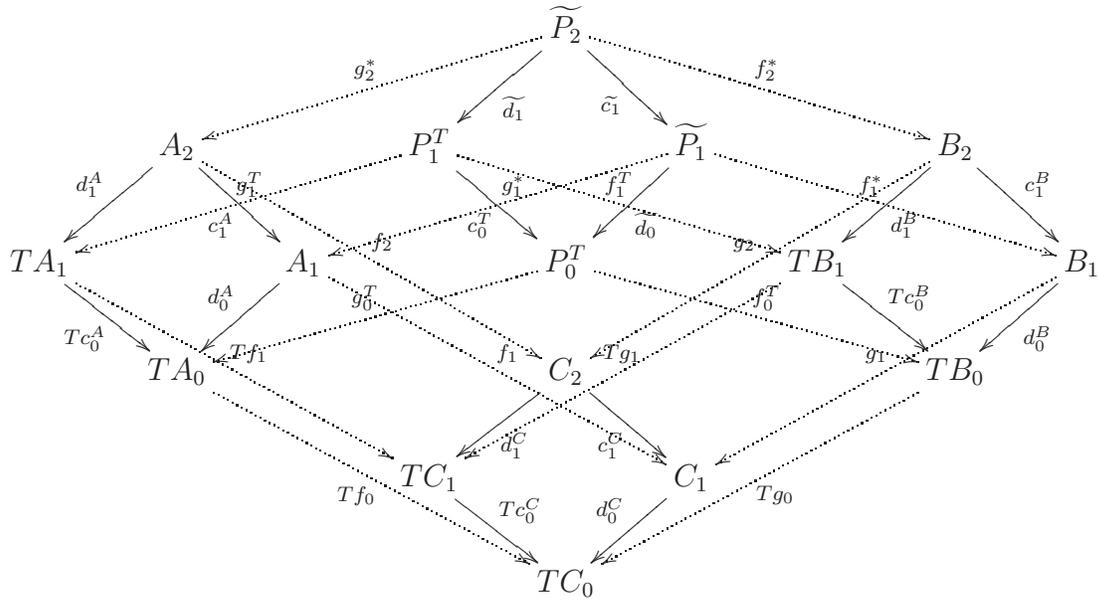
\begin{center}
$\xymatrix{ &&&& \widetilde{P_2}\ar@{..>}[llld]_{g_2^*}\ar[ld]^{\widetilde{d_1}}\ar[rd]_{\widetilde{c_1}}\ar@{..>}[rrrd]^{f_2^*} &&&&\\
& A_2\ar[ld]_{d_1^A}\ar[rd]_{c_1^A}\ar@{..>}[rrrdd]^{f_2} && P_1^T\ar@{..>}[llld]_{\qquad g_1^T}\ar[rd]_{c_0^T}\ar@{..>}[rrrd]^{f_1^T\quad} && \widetilde{P_1}\ar@{..>}[llld]_{\qquad g_1^*}\ar[ld]^{\widetilde{d_0}}\ar@{..>}[rrrd]^{f_1^*\qquad} && B_2\ar@{..>}[llldd]_{g_2\;}\ar[ld]^{d_1^B}\ar[rd]^{c_1^B} &\\
TA_1\ar[rd]_{Tc_0^A}\ar@{..>}[rrrdd]^{Tf_1} && A_1\ar[ld]_{d_0^A}\ar@{..>}[rrrdd]^{f_1} && P_0^T\ar@{..>}[llld]_{g_0^T}\ar@{..>}[rrrd]^{f_0^T} && TB_1\ar@{..>}[llldd]_{\qquad Tg_1}\ar[rd]^{Tc_0^B} && B_1\ar@{..>}[llldd]_{g_1}\ar[ld]^{d_0^B}\\
& TA_0\ar@{..>}[rrrdd]_{Tf_0} &&& C_2\ar[ld]^{d_1^C}\ar[rd]_{c_1^C} &&& TB_0\ar@{..>}[llldd]^{Tg_0} &\\
&&& TC_1\ar[rd]^{Tc_0^C} && C_1\ar[ld]_{d_0^C} &&&\\
&&&& TC_0 &&&&\\
}$	
\end{center}

 \medskip

 {\em STEP 3: Comparing $\widetilde{P_2}$ with $P_2$, and $P_0^{TT}:=TTA_0\times_{TTC_0}TTB_0$ with $TTP_0$ and $P_0^T$.}\\
Taking the morphism $\ell_0':P_1\longrightarrow \widetilde{P_1}$ to the next level, one first defines the morphism $\ell_1:TP_1\longrightarrow P_1^T$ by the left cube below, then routinely confirms the equality $c_0^T\cdot \ell_1=\ell_0\cdot Tc_0^P$, and finally uses Step 2 to define $\ell_1':=\ell_1\times\ell_0':P_2\longrightarrow\widetilde{P_2}$ as the morphism fitting into the cube on the right.

 \begin{center}

 $\xymatrix{
& T\widetilde{P_1}\ar[ld]_{Tg_1^*}\ar[dd]_{Tf_1^*} && TP_1\ar[ll]_{T\ell_0'}\ar@{-->}[ld]_{\ell_1}\ar[dd]^{T\ell_o'}\\
TA_1\ar[dd]_{Tf_1} && P_1^T\ar[ll]_{\qquad\quad g_1^T}\ar[dd]_{f_1^T} &\\
& TB_1\ar[ld]_{Tg_1} && T\widetilde{P_1}\ar[ll]_{\qquad\quad Tf_1^*}\ar[ld]^{Tf_1^*}\\
TC_1 && TB_1\ar[ll]_{Tg_1} &\\
}$	
\qquad
$\xymatrix{
& TP_1\ar[ld]_{\ell_1}\ar[dd]_{Tc_0^P} && P_2\ar[ll]_{d_1^P}\ar@{..>}[ld]_{\ell_1'}\ar[dd]^{c_1^P}\\
P_1^T\ar[dd]_{c_0^T} && \widetilde{P_2}\ar[ll]_{\qquad\quad \widetilde{d_1}}\ar[dd]_{\widetilde{c_1}} &\\
& TP_0\ar[ld]_{\ell_0} && P_1\ar[ll]_{\qquad\quad d_0^P}\ar[ld]^{\ell_0'}\\
P_0^T && \widetilde{P_1}\ar[ll]_{\widetilde{d_0}} &\\
}$	
\end{center}

 One level up, the counterpart of the comparison morphism $\ell_0$ is the canonical morphism
$$\ell_0^T: TTP_0=TT(A_0\times B_0)\longrightarrow P_0^{TT}:=TTA_0\times_{TTC_0}TTB_0.$$
It turns out to cooperate well with the morphisms $$d_1^T:=Td_0^A\times Td_0^B:P_1^T\longrightarrow P_0^{TT}\quad\text{and}\quad\widetilde{\mu}:=\mu_{A_0}\times\mu_{B_0}:P_0^{TT}\longrightarrow P_0^T$$ that are described by the following commutative diagram, the front, center and back panels of which are pullbacks:
\begin{center}
$\xymatrix{&& TA_1\ar[ldd]_{Td_0^A}\ar[d]^{Tf_1} && P_1^T\ar[ll]_{g_1^T}\ar@{-->}[ldd]_{d_1^T\!}\ar[d]^{f_1^T}\\
&& TC_1\ar[ldd]^{\!\!Td_0^C} && TB_1\ar[ll]_{Tg_1\qquad}\ar[ldd]^{Td_0^B}\\
& TTA_0\ar[ldd]_{\mu_{A_0}}\ar[d]^{TTf_0} && P_0^{TT}\ar[ll]_{\qquad\quad g_0^{TT}}\ar@{-->}[ldd]_{\widetilde{\mu}}\ar[d]_{f_0^{TT}\!\!}\\
& TTC_0\ar[ldd]^{\mu_{C_0}} && TTB_0\ar[ll]_{TTg_0\quad}\ar[ldd]^{\mu_{B_0}} &\\
TA_0\ar[d]_{Tf_0} && P_0^T\ar[ll]_{\qquad g_0^T}\ar[d]_{f_0^T} &&\\
TC_0 && TB_0\ar[ll]_{Tg_0} &&\\
}$	
\end{center}
Indeed, one routinely confirms the equalities
$$d_1^T\cdot \ell_1=\ell_0^T\cdot Td_0^P\quad\text{and}\quad \widetilde{\mu}\cdot \ell_0^T=\ell_0\cdot\mu_{P_0}.$$
For later use we also note that, having obtained the morphisms
$$g_2':=Tg_1'\times g_1':P_2\longrightarrow A_2\quad\text{and}\quad f_2':= Tf_1'\times f_1':P_2\longrightarrow B_2$$
in the standard fashion, one easily confirms the equalities
$$g_2'=g_2^*\cdot\ell_1'\quad\text{and}\quad f_2'=f_2^*\cdot \ell_1'.$$

\medskip
{\em STEP 4: Defining $m^P:P_2\longrightarrow P_1$.}\\
Just like the definition of $i^P$ needed the precursor $\widetilde{i}$, we first consider the morphism $\widetilde{m}:= m^A\times m^B$ as described by the cube below on the left, whose front and back panels are pullbacks:
\begin{center}
$\xymatrix{& A_2\ar[ld]_{m^A}\ar[dd]_{f_2} && \widetilde{P_2}\ar[ll]_{g_2^*}\ar@{-->}[ld]_{\widetilde{m}}\ar[dd]^{f_2^*}\\
A_1\ar[dd]_{f_1} && \widetilde{P_1}\ar[dd]_{f_1^*}\ar[ll]_{\qquad \quad g_1^*} &\\
& C_2\ar[ld]_{m^C\!\!} && B_2\ar[ll]_{\qquad\quad g_2}\ar[ld]^{m^B}\\
C_1 && B_1\ar[ll]_{g_1}\\
}$
\quad
$\xymatrix{& TTP_0\ar[ld]_{\mu_{P_0}}\ar[dd]_{\ell_0^T} && P_2\ar[ll]_{Td_0^P\cdot d_1^P}\ar@{..>}[ld]_{m^P}\ar[dd]^{\ell_1'}\\
TP_0\ar[dd]_{\ell_0} && P_1\ar[ll]_{\qquad\quad d_0^P}\ar[dd]^{\ell_0'} &\\
& P_0^{TT}\ar[ld]_{\ell_0} && \widetilde{P_2}\ar[ll]_{d_1^T\cdot\widetilde{d_1}\qquad\quad}\ar[ld]^{\widetilde{m}}\\
P_0^T && \widetilde{P_1}\ar[ll]_{\widetilde{d_0}} &\\
}$	
\end{center}	
With the morphisms established in Step 3 one routinely shows the equality
$$\widetilde{d_0}\cdot\widetilde{m}=\widetilde{\mu}\cdot d_1^T\cdot\widetilde{d_1},$$
which then gives
$$\widetilde{d_0}\cdot\widetilde{m}\cdot\ell_1'=\widetilde{\mu}\cdot d_1^T\cdot\widetilde{d_1}\cdot\ell_1'=\widetilde{\mu}\cdot d_1^T\cdot\ell_1\cdot d_1^P=\widetilde{\mu}\cdot d_0^T\cdot Td_0^P\cdot d_1^P=\ell_0\cdot\mu_{P_0}\cdot Td_0^P\cdot d_1^P.$$
Consequently we obtain the desired morphism $m^p$ making the cube above on the right commute. By definition, it cooperates correctly with the domain morphisms. To show the corresponding statement for the codomain morphisms, one first verifies easily the equality
$$ \widetilde{c_0}\cdot\widetilde{m}=\widetilde{c_0}\cdot\widetilde{c_1} $$
and then concludes
$$c_0^P\cdot m^P=\widetilde{c_0}\cdot \ell_0'\cdot m^P=\widetilde{c_0}\cdot\widetilde{m}\cdot\ell_1'=\widetilde{c_0}\cdot\widetilde{c_1}\cdot\ell_1'=\widetilde{c_0}\cdot\ell_0'\cdot c_1^P=c_0^P\cdot c_1^P\;.$$
	
\medskip

 {\em STEP 5: Verifying the equalities $m^P\cdot i_1^P=1_{P_1}=m^P\cdot i_2^P$.}\\
Similarly to the definition of $\widetilde{i}, \widetilde{\mu}$ one first defines the auxiliary morphisms
$$ \widetilde{i_1}=i_1^A\times i_1^B, \quad\widetilde{i_2}=i_2^A\times i_2^B:\;\widetilde{P_1}\longrightarrow\widetilde{P_2},\quad \widetilde{\eta}=T\eta_{A_0}\times\ T\eta_{B_0}:P_0^T\longrightarrow P_0^{TT}.$$
With the readily verified equalities
$$\ell_1'\cdot i_1^P=\widetilde{i_1}\cdot \ell_0',\quad\ell_1'\cdot i_2^P=\widetilde{i_2}\cdot \ell_0',\quad\ell_0^T\cdot T\eta_{P_0}=\widetilde{\eta}\cdot\ell_0,\quad d_1^T\cdot\widetilde{d_1}\cdot\widetilde{i_1}=\widetilde{\eta}\cdot \widetilde{d_0}$$
one confirms that these morphisms are the correct ``approximations" of $i_1^P, i_2^P$ and $T\eta_{P_0}$,
and they also satisfy
$$\widetilde{m}\cdot\widetilde{i_1}=1_{\widetilde{P_1}}=\widetilde{m}\cdot\widetilde{i_2}\quad\text{and}\quad\widetilde{\mu}\cdot\widetilde{\eta}=1_{P_0^T}. $$
With the first set of equalities one sees that the following diagram commutes without the curved arrows, whilst the second set shows that it commutes with the curved arrows as well. Hence, with the front face being a pullback, one concludes $m^P\cdot i_1^P= 1_{P_1}$, as desired.
\begin{center}
$\xymatrix{&& TP_0\ar[ldd]_{T\eta_{P_0}}\ar@/_32
pt/[lldddd]_1\ar[d]^{\ell_0} && P_1\ar[ll]_{d_0^P}\ar[ldd]_{\!i_1^P\!}\ar[d]^{\ell_0'}\\
&& P_0^T\ar[ldd]^{\!\!\widetilde{\eta}}\ar@/^20pt/[lldddd]^1 && \widetilde{P_1}\ar[ll]_{\widetilde{d_0}\qquad}\ar[ldd]^{\widetilde{i_1}}\ar@/^20pt/[lldddd]^1\\
& TTP_0\ar[ldd]_{\mu_{P_0}}\ar[d]^{\ell_0^T} && P_2\ar[ll]_{\qquad\quad Td_0^P\cdot d_1^P}\ar[ldd]_{m^P\!}\ar[d]_{\ell_1'\!\!}\\
& P_0^{TT}\ar[ldd]^{\widetilde{\mu}} && \widetilde{P_2}\ar[ll]_{d_1^T\cdot\widetilde{d_1}\quad}\ar[ldd]^{\widetilde{m}} &\\
TP_0\ar[d]_{\ell_0} && P_1\ar[ll]_{\qquad d_0^P}\ar[d]_{\ell_0'} &&\\
P_0^T && \widetilde{P_1}\ar[ll]_{\widetilde{d_0}} &&\\
}$	
\end{center}
The other needed equality, $m^P\cdot i_2^P = 1_{P_1}$, follows more succinctly from
$$d_0^P\cdot m^P\cdot i_2^P=\mu_{P_0}\cdot Td_0^P\cdot d_1^P\cdot i_2^P=\mu_{P_0}\cdot Td_0^P\cdot\eta_{P_1}=\mu_{P_0}\cdot\eta_{TP_0}\cdot d_0^P=d_0^P,$$
$$\ell_0'\cdot m^P\cdot i_2^P= \widetilde{m}\cdot\ell_1'\cdot i_2^P=\widetilde{m}\cdot\widetilde{i_2}\cdot\ell_0'=\ell_0'\;.$$

 \medskip
{\em STEP 6: Revisiting Steps 2 and 3, one level up.}\\
Raising the indices of all objects and morphisms in the defining diagrams for $c_0^T,\,\widetilde{d_1},\,\widetilde{c_1}$ of Step 2 by 1, we define the morphisms
$$c_1^T:P_2^T\longrightarrow P_1^T,\quad \widetilde{d_2}:\widetilde{P_3}\longrightarrow P_2^T,\quad\widetilde{c_2}: \widetilde{P_3}\longrightarrow\widetilde{P_2},$$
including their domains, which come with projections $g_2^T, f_2^T$ and $g_3^*,f_3^*$, respectively.
As in Step 2 one obtains the pullback diagram
\begin{center}
$\xymatrix{P_2^T\ar[d]_{c_1^T} & \widetilde{P_2}\ar[l]_{\widetilde{d_2}}\ar[d]^{\widetilde{c_2}}\\
P_1^T & P_2^T\ar[l]_{\widetilde{d_1}}\\
}$	
\end{center}
One then defines the morphisms
$$\ell_2: TP_2\longrightarrow P_2^T\quad\text{and}\quad\ell_2':P_3\longrightarrow\widetilde{P_2}$$
just like $\ell_1, \ell_1'$ have been defined in Step 3, again by just raising all indices of the defining diagrams by 1. We record the equalities $\widetilde{d_1}\cdot \ell_2=\ell_1\cdot d_1^P$ and $c_1^T\cdot \ell_2=\ell_1\cdot Tc_1^P$; they, together with the above pullback, facilitate the definition of $\ell_2'$, which is determined by
$$\widetilde{d_2}\cdot\ell_2'=\ell_2\cdot d_2^P\quad\text{and}\quad \widetilde{c_2}\cdot\ell_2'=\ell_1'\cdot c_2^P.$$

\medskip
{\em STEP 7: Verifying the equality $m^P\cdot m_1^P=m^P\cdot m_2^P$.}
Analogously to the procedure pursued in Step 5, we first consider the auxiliary morphisms
$$ \widetilde{m_1}=m_1^A\times m_1^B, \quad\widetilde{m_2}=m_2^A\times m_2^B:\;\widetilde{P_3}\longrightarrow\widetilde{P_2},\quad m^T=Tm^A\times\ Tm^B:P_2^T\longrightarrow P_1^{T}.$$
One obtains immediately the ``approximate'' version of the desired equality, that is
$$\widetilde{m}\cdot\widetilde{m_1}=\widetilde{m}\cdot\widetilde{m_2}.$$
But one also needs to establish successively the equalities
$$d_1^T\cdot m^T\cdot\ell_2=\ell_0^T\cdot Td_0^P\cdot Tm^P,\;d_1^T\cdot m^T\cdot\widetilde{d_2} =d_1^T\cdot\widetilde{d_1}\cdot \widetilde{m_1},\;\ell_1'\cdot m_1^P=\widetilde{m_1}\cdot \ell_2',\;\ell_1'\cdot m_2^P=\widetilde{m_2}\cdot\ell_2',$$
in order to be able to derive the desired equality, as follows:
\begin{align*}
d_0^P\cdot m^P\cdot m_1^P &=\mu_{P_0}\cdot Td_0^P\cdot d_1^P\cdot m_1^P &\ell_0'\cdot m^P\cdot m_1^P&=\widetilde{m}\cdot \ell_1'\cdot m_1^P\\
&=\mu_{P_0}\cdot Td_0^P\cdot Tm^P\cdot d_2^P &&=\widetilde{m}\cdot \widetilde{m_1}\cdot\ell_2'\\
&=\mu_{P_0}\cdot \mu_{TP_0}\cdot TTd_0^P\cdot Td_1^P\cdot d_2^P &&=\widetilde{m}\cdot \widetilde{m_2}\cdot\ell_2'\\
&=\mu_{P_0}\cdot Td_0^P\cdot\mu_{P1}\cdot Td_1^P\cdot d_2^P &&=\widetilde{m}\cdot\ell_1'\cdot m_2^P\\
&=\mu_{P_0}\cdot Td_0^P\cdot d_1^P\cdot m_2^P &&=\ell_0'\cdot m^P\cdot m_2^P. \\
&=d_0^P\cdot m^P\cdot m_2^P\\
%&\\
%\ell_0'\cdot m^P\cdot m_1^P&=\widetilde{m}\cdot \ell_1'\cdot m_1^P\\
%&=\widetilde{m}\cdot \widetilde{m_1}\cdot\ell_2'\\
%&=\widetilde{m}\cdot \widetilde{m_2}\cdot\ell_2'\\
%&=\widetilde{m}\cdot\ell_1'\cdot m_2^P\\
%&=\ell_0'\cdot m^P\cdot m_2^P.
\end{align*}

 \medskip{\em STEP 8: Verifying the universal property.}\\
By design, $g'=(g_0',g_1')$ and $ f'=(f_0',f_1')$ coo
perate with the domain and codomain morphisms and satisfy $g'\cdot f=f'\cdot g$. Their $\mathbb T$-functoriality follows instantaneously as well:
$$ g_1'\cdot i^P= g_1^*\cdot l_0'\cdot i^P=g_1^*\cdot\widetilde{i}=i^A\cdot g_0'\;;$$
$$ g_1'\cdot m^P= g_1^*\cdot \ell_0'\cdot m^P=g_1^*\cdot\widetilde{m}\cdot\ell_1'=m^A\cdot g_2^*\cdot \ell_1'=m^A\cdot g_2'\;;$$
and, likewise, $f_1'\cdot i^P=i^B\cdot f_0',\;f_1'\cdot m^P=m^B\cdot f_2'$.

 \medskip
Considering $\mathbb T$-functors $s:X\longrightarrow A$ and $t:X\longrightarrow B$ with $f\cdot s= g\cdot t$ one obtains the morphisms $$u_0: X_0\longrightarrow P_0\quad\text{and}\quad u_1^*: X_1\longrightarrow\widetilde{ P_1}$$
satisfying $g_0'\cdot u_0=s_0,\; f_0'\cdot u_0=t_0$ and $g_1^*\cdot u_1^*= s_1,\;f_1^*\cdot u_1^*=t_1$.
With the pullback property of $P_0^T$ one confirms the equality $\widetilde{d_0}\cdot u_1^*==\ell_0\cdot Tu_0\cdot d_0^X$ which, with the pullback property of $P_1$, determines the morphism $$u_1:X_1\longrightarrow P_1$$
with $d_0^P\cdot u_1=Tu_0\cdot d_0^X$ and $ \ell_0'\cdot u_1= u_1^*$. One routinely confirms the equality $c_0^P\cdot u_1=u_0\cdot c_0^X$, and $g_1'\cdot u_1=s_1,\; f_1'\cdot u_1 =t_1$ holds trivially. From
$$g_1'\cdot i^P\cdot u_0=g_1^*\cdot \ell_o'\cdot i^P\cdot u_0=g_1^*\cdot \widetilde{i}\cdot u_0=i^A\cdot g_0'\cdot u_0=i^A\cdot s_0=s_1\cdot i^X=g_1'\cdot u_1\cdot i^X$$
and $f_1'\cdot i^P\cdot u_0=f_1'\cdot u_1\cdot i^X$ one concludes $u_1\cdot i^X=i^P\cdot u_0$. Finally, establishing first the equalities $g_2'\cdot u_2=s_2,\, f_2'\cdot u_2= t_2$ one similarly confirms the remaining $\mathbb T$-functoriality condition for $u$, that is: $u_1\cdot m^X=m^P\cdot u_2$.

 Any $\mathbb T$-functor $v:X\longrightarrow P$ with $g'\cdot v=g'\cdot u, \;f'\cdot u=f'\cdot v$ is easily seen to have to satisfy $\ell_0'\cdot v_1=\ell_0'\cdot u_1$ and, hence $v_1=u_1$, whilst $v_0=u_0$ holds trivially.
\end{proof}

 As already explained before Proposition \ref{pbstability}, the now confirmed pullback stability of ${\sf Perf}(\mathbb T)$, in combination with Theorem \ref{reflective} and Proposition \ref{factsystems}, gives us the principal result of this paper:

 \begin{theorem}\label{compfact}(Comprehensive Factorization Theorem) Let $\mathbb T=(T,\eta,\mu)$ be a monad on a category $\C$ with pullbacks and coequalizers of reflexive pairs, such that these coequalizers are stable under pullback and preserved by $T$. Then every $\mathbb T$-functor factors into an initial $\mathbb T$-functor followed by a perfect one.
	\end{theorem}
	
	Here, by definition, a $\mathbb T$-functor is {\em initial} if it is left orthogonal to the class ${\sf Perf}(\mathbb T)$ in ${\sf Cat}(\mathbb T)$ (as defined in the proof of Proposition \ref{factsystems}).
	
	Revisiting the construction of pullbacks as in Proposition \ref{pbstability}, one sees easily that the full subcategory ${\sf Ord}(\mathbb T)$ is closed under the formation of pullbacks in ${\sf Cat}(\mathbb T)$. One also confirms immediately that, for every perfect $\mathbb T$-functor $p:P\longrightarrow B$, the domain $P$ is ordered when $B$ is ordered. Consequently, Theorem \ref{compfact} gives us also the comprehensive factorization system for ordered $\mathbb T$-categories:
	
	\begin{corollary}
	Under the hypotheses	 of Theorem {\em \ref{compfact}}, the comprehensive factorization of $\mathbb T$-functor of ordered $\mathbb T$-categories stays in the category ${\sf Ord}(\mathbb T)$.
	\end{corollary}
	
	\begin{remark}
	After being embedded into ${\sf Ord}(\mathbb T)$, every morphism in ${\sf Alg}(\mathbb T)$ is, by definition, trivially perfect. Hence, the comprehensive factorization system becomes the trivial system (isomorphisms, all morphisms) when restricted to ${\sf Alg}(\mathbb T)$.
	\end{remark}	

 \begin{examples}\label{examplefacts}
(1)	For $\mathbb T$ the identity monad on a category $\C$ with pullbacks and pullback-stable coequalizers of reflexive pairs, Theorem \ref{compfact} reproduces the known comprehensive factorization system in the category of internal categories in $\C$; see \cite{Johnstone2002}. For $\C={\bf Set}$, the notion of initial $\mathbb T$-functor, defined to be left-orthogonal to the class of perfect $\mathbb T$-functors, returns the standard notion of initial functor of ordinary (small) categories. Recall that a functor $E:\I\longrightarrow \J$ is initial if it leaves all $\J$-indexed limits in any category invariant, so that for all $D:\J\longrightarrow\A$ the limit of $D$ (if it exists) serves also as the limit of $DE$ (making the canonical ${\rm lim}D\longrightarrow {\rm lim}DE$ an isomorphism). They are characterized by the property that, for every object $j\in\J$, the comma category $E/j$ is (non-empty and) connected.

 (2) The endofunctor of the list monad $\mathbb L$ on ${\bf Set}$ preserves pullbacks and coequalizers, and they are stable under pullback. Theorem \ref{compfact} guarantees the existence of the comprehensive factorization system in ${\bf MulCat}={\sf Cat}(\mathbb L)$. It extends the comprehensive factorization system in $\bf Cat$ and has been described in \cite{BergerKaufmann2017}, as a particular example of a general environment for which the authors establish comprehensive factorization systems. That environment, however, does not seem to include Burroni's $\mathbb T$-categories.

 (3) The endofunctor of the ultrafilter monad $\mathbb U$ on ${\bf Set}$ may fail to preserve the coequalizer of a reflexive pair (and it does not preserve pullbacks either), as shown by the following example, which is due to Dirk Hofmann (private communication, November 2020):\\
\indent Consider the set $R=\{(n,m)\;|\;|n-m|\le 1\}$ of pairs of equal or adjacent natural numbers, with its projections $p_1, p_2$ to $\mathbb N$. Their coequalizer in $\bf Set$ identifies all numbers. It now suffices to show that, by contrast, the maps $Up_1, Up_2: UR\to U\mathbb N=\beta\mathbb N$ cannot identify a fixed ultrafilter with a free ultrafilter on $\mathbb N$. Indeed, if we had an ultrafilter $\mathfrak r$ on $R$ with $Up_1(\mathfrak r)=\dot{n}$ fixed and $Up_2(\mathfrak r)$ free, then we would have $p_1^{-1}(n)\in\mathfrak r$. Since $p_1^{-1}(n)$ has at most 3 elements, this would force $\mathfrak r$ to be fixed. But then also $Up_2(\mathfrak r)$ would have to be fixed.\\
\indent Nevertheless, the comprehensive factorization of a continuous map $f:X\longrightarrow Y$, known as its (antiperfect, perfect)-factorization, exists in large subcategories of ${\bf Top}={\sf Ord}(\mathbb U)$, such as the category of Tychonoff spaces. A continuous map of Tychonoff spaces is perfect if, and only if, its naturality square given by the Stone-$\rm\check{C}$ech compactification is a pullback diagram (so that its Stone-$\rm{\check{C}}$ech extension maps the remainder of the compactification of its domain into the remainder of the compactification of its codomain), and it is antiperfect (= initial) precisely when its Stone-$\rm{\check{C}}$ech extension is a homeomorphism. These facts have been established much more generally in an abstract category that comes with a distinguished class of ``closed morphisms" satisfying suitable properties: see \cite{Tholen1999} and \cite{ClementinoGiuliTholen2004}. In the settings of theses papers the construction of the factorization largely follows the lead of \cite{Ringel1970} and \cite{CassidyHebertKelly1985}, as indicated in the Introduction.
\end{examples}

 \refs
\bibitem[Ad\'amek, Herrlich, Strecker, 1990]{AHS} J. Ad\'amek, H. Herrlich, and G.E. Strecker, \emph{Abstract and Concrete Categories}, John Wiley \& Sons, New York, 1990. Republished in: Reprints in Theory and Applications of Categories 17, 2006.
\bibitem[Barr, 1970]{Barr1970} M. Barr, {\em Relational algebras}, in: Lecture Notes in Mathematics 137, pp. 39--55, Springer-Verlag, New York, 1970.
\bibitem[Barr, Wells, 1985]{BarrWells1985} M. Barr and C. Wells, \emph{Toposes, Triples and Theories}, Springer-Verlag, New York, 1985.
\bibitem[Bastiani, Ehresmann, 1969]{BastianiEhresmann1969} A. Bastiani, C. Ehresmann, {\em Cat\'egories de foncteurs structur\'es}, Cahiers de Topologie et G\'eom\'etrie Diff\'erentielle Cat\'egoriques, 11:329--384, 1969.
\bibitem[Berger, Kaufmann, 2017]{BergerKaufmann2017} C. Berger and R.M. Kaufmann, {\em Comprehensive factorisation systems}, Tbilisi Mathematical Journal, 10(3):255--277, 2017.
\bibitem[Bourbaki, 1989]{Bourbaki1989} N. Bourbaki, \emph{General Topology, Chapters 1--4}, Springer-Verlag, Berlin, 1989.
\bibitem[Burroni, 1971]{Burroni1971} A. Burroni, \emph{T-cat\'egories (cat\'egories dans un triple)}, Cahiers de Topolologie et G\'eom\'etrie Diff\'erentielle, 12: 215-321, 1971.
\bibitem[Cassidy, H\'{e}bert, Kelly, 1985]{CassidyHebertKelly1985} C. Cassidy, M. H\'{e}bert and G.M. Kelly, {\em Reflective subcategories, localizations, and factorization systems}, Journal of the Australian Mathematical Society (Ser. A), 38:387--429, 1985.
\bibitem[Clementino, Giuli, Tholen, 2004]{ClementinoGiuliTholen2004} M.M Clementino, E. Giuli and W. Tholen, {\em A Functional Approach to General Topology}, in: M.C. Pedicchio and W. Tholen (editors), {\em Categorical Foundations}, Cambridge University Press, Cambridge 2004.
\bibitem[Clementino, Hofmann, 2003]{ClementinoHofmann2002} M.M. Clementino and D. Hofmann, {\em Topological features of lax algebras}, Applied Categorical Structures, 11:267--286, 2003.
\bibitem[Clementino, Tholen 2003]{ClementinoTholen2003}	M.M. Clementino and W. Tholen, {\em Metric, topology and multicategory---a common approach}, Journal of Pure and Applied Algebra, 179:13-47, 2003.
\bibitem[Cruttwell, Shulman, 2010]{CruttwellShulman2010} G.S.H. Cruttwell and M.A. Shulman, {\em A unified framework for generalized multicategories}, Theory and Applications of Categories, 24(21):580--655, 2010.
\bibitem[Dawson, Par\'e, Pronk, 2010] {DawsonParePronk2010} R. Dawson, R. Par\'e and D. Pronk,{\em The span construction}, Theory and Applications of Categories, 24(13):301--377, 2010.
\bibitem[Dikranjan, Tholen, 1995]{DikranjanTholen1995} D. Dikranjan and W. Tholen, {\em Categorical Structure of Closure Operators}, Kluwer Academic Publishers, Dordrecht, 1995.
\bibitem[Ehrbar, Wyler, 1968]{EhrbarWyler1968} H. Ehrbar and O. Wyler, {\em On subobjects and images in categories}, Report
68-34, Carnegie Mellon University, 28 pp., 1968.
\bibitem[Engelking, 1989]{Engelking1989} R. Engelking, {\em General Topology} (2nd edition), Heldermann, Berlin, 1989.
\bibitem[Freyd, Kelly, 1972]{FreydKelly1972} P.J Freyd and G.M. Kelly, {\em Categories of continuous functors, I}, Journal of Pure and Applied Algebra, 2:169--191, 1972.
\bibitem[Giuli, Tholen, 2007]{GiuliTholen2007} E. Giuli and W. Tholen, \emph{A topologist's view of Chu spaces}, Applied Categorical Structures, 15: 573-598, 2007.
\bibitem[Gray, 1969]{Gray1969} J.W. Gray, {\em The categorical comprehension scheme}, in: Lecture Notes in Mathematics 99, pp. 242--312, Springer-Verlag, Berlin ,1969.
\bibitem[Grothendieck, 1961]{Grothendieck1961} A. Grothendieck, {\em Cat\'{e}gories Fibr\'{e}es et Descente}, S\'{e}minaire de G\'{e}om\'{e}trie Alg\'{e}brique, Institute des Hautes \'{E}tudes Scientifiques, Paris, 1961.
\bibitem[Henriksen, Isbell, 1958]{HenriksenIsbell1958} M. Henriksen and J.R. Isbell, {\em Some properties of compactifications}, Duke Mathematical Journal, 25:83--106, 1958.
\bibitem[Hermida, 2000]{Hermida2000} C. Hermida, {\em Representable multicategories}, Advances in Mathematics, 151:164--225, 2000.
\bibitem[Herrlich, 1972]{Herrlich1972} H. Herrlich, {\em A generalization of perfect maps}, in: J. Nov\'{a}k (editor), {\em General Topology and its Relations to Modern Analysis and Algebra}, Proceedings of the Third Prague Topological Symposium, 1971, pp. 187--191, Academia Publishing House of the Czechoslovak Academy of Sciences, Prague, 1972.
\bibitem[Hofmann, Seal, Tholen, 2014]{HST} D. Hofmann, G.J. Seal and W. Tholen (editors): \emph{Monoidal Topology, A Categorical Approach to Order, Metric, and Topology}, Cambridge University Press, New York 2014.	
\bibitem[Hoffmann, 1972]{Hoffmann1972} R.-E. Hoffmann, {\em Die kategorielle Auffassung der Inoitial- und Finaltopologie}, Doctoral thesis, Ruhr-Universi\"{a}t, Bochum, 1972.
\bibitem[James, 1989]{James1989} I.M. James, \emph{Fibrewise Topology}, Cambridge University Press, Cambridge, 1989.
\bibitem[Johnstone, 2002]{Johnstone2002} P.T. Johnstone, {\em Sketches of an Elephant, A Topos Theory Compendium}, Oxford University Press, Oxford, 2002.
\bibitem[Lambek, 1969]{Lambek1969} J. Lambek, {\em Deductive systems and categories, II. Standard constructions and closed categories}, in: Lecture Notes in Mathematics 86, pp. 76--122, Springer-Verlag, New York, 1969.
\bibitem[Lawvere, 1970]{Lawvere1970} F.W. Lawvere, {\em Equality in hyperdoctrines and comprehension schema as an adjoint functor}, in: A. Heller (editor), {\em Proceedings of the New York Symposium on Applications of Categorical Algebra}, pp. 1--14, American Mathematical Society, Providence RI, 1970.
\bibitem[Lawvere, 1973]{Lawvere1973} F.W. Lawvere, {\em Metric spaces, generalized logic, and closed categories}, Rendiconti Sem. Mat. Fis. Milano 43:135--166, 1973. Republished in: Reprints in Theory and Applications of Categories 1, 2002.
\bibitem[Leinster, 2004]{Leinster2004} T. Leinster, {\em Higher Operads, Higher Categories}, Cambridge University Press, Cambridge, 2004.
\bibitem[Lowen, 1997]{Lowen1997} R. Lowen, {\em Approach Spaces: The Missing Link in the Topology-Uniformity-Metric Triad}, Oxford University Press, Oxford, 1997.
\bibitem[MacDonald, Tholen, 1982]{MacDonaldTholen1982} J. MacDonald and W. Tholen, {\em Decomposition of morphisms into infinityely many factors}, in: Lecture Notes in Mathematics 962, pp. 175--189, Springer-Verlag, New York, 1982.
\bibitem[Mac Lane, 1998]{Mac Lane} S. Mac Lane, \emph{Categories for the Working Mathematician}, 2nd edition, Springer-Verlag, New York, 1998.
\bibitem[Manes, 1969]{Manes1969} E.G. Manes, {\em A triple theoretic construction of compact algebras}, in: Lecture Notes in Mathematics 80, pp. 91--118, Springer-Verlag, New York, 1969.
\bibitem[Ringel, 1970]{Ringel1970} C.M. Ringel, {\em Diagonalisierungspaare, I}, Mathematische Zeitschrift 117:249--266, 1970.
\bibitem[Rosebrugh, Wood, 2002]{RosebrughWood2002} R. Rosebrugh and R.J. Wood, {\em Distributive laws and factorization}, Journal of Pure and Applied Algebra, 175:327--353, 2002.
\bibitem[Street, Walters 1973]{StreetWalters1973} R. Street and R.F.C. Walters, {\em The comprehensive factorization of a functor}, Bulletin of the American Mathematical Society, 79(5):936--941, 1973.
\bibitem[Tholen, 1979]{Tholen1979} W. Tholen, {\em Semi-topological functors, I}, Journal of Pure and Applied Algebra, 15:53--73, 1979.
\bibitem[Tholen, 1999]{Tholen1999} W. Tholen, {\em A categorical guide to separation, compactness and perfectness}, Homology, Homotopy and Applications 1:147--161, 1999.
\bibitem[Tholen, Yeganeh, 2021]{TholenYeganeh2021} W. Tholen and L. Yeganeh, {\em ${\sf Cat}(\mathbb T)$ as a closed category} (in preparation).
\bibitem[Whyburn 1966]{Whyburn1966} G.T. Whyburn, {\em Compactification of mappings}, Mathematische Annalen, 166:168-174, 1966.
\bibitem[Zawadowski, 2011]{Zawadowski2011} M. Zawadowski, {\em Lax monoidal fibrations}, in: {\em Models, Logics, and Higher-Dimensional Categories: A Tribute to the Work of Mihaily Makkai}, CRM Proceedings and Lecture Notes 53, pp 341--426, American Mathematical Society, Providence RI, 2011.

\end{document}